\documentclass[12pt,a4paper]{amsart}
\usepackage{amsmath,mathtools,mathrsfs, multirow}

\calclayout
\usepackage{fancyhdr}
\usepackage{stmaryrd}
\usepackage{comment}
\usepackage{graphicx}
\usepackage{amsfonts}
\usepackage{amssymb}
\usepackage{wrapfig,lipsum,booktabs}
\usepackage{placeins}
\usepackage[hang]{subfigure}
\usepackage[utf8]{inputenc}
\usepackage[english]{babel}
\usepackage{amsthm}
\usepackage{subfigure} 
\usepackage[utf8]{inputenc}
\usepackage[margin=0.79in]{geometry} 
\usepackage[final]{pdfpages}
\newtheorem{theorem}{Theorem}
\newtheorem*{theorem*}{Theorem}
\newtheorem{remark}{Remark}
\newtheorem{lemma}{Lemma}
\theoremstyle{definition}

\theoremstyle{notation}

\theoremstyle{corollary}
\newtheorem{corollary}{Corollary}

\newtheorem{case}{Case}

\usepackage{times}
\usepackage{blindtext}
\usepackage{setspace}
\usepackage{enumitem}
\usepackage{hyperref}
\title[Capparelli's partition theorem as part of an infinite hierarchy]{Capparelli's partition theorem as part of an infinite hierarchy:\\ Combinatorial and Weighted Words extensions of recent work}
\author{Yazan Alamoudi}
\address{Department of Mathematics, University of Florida, Gainesville,
FL 32611, USA}
\email{yazanalamoudi(at)ufl.edu}
\author{Krishnaswami Alladi}
\address{Department of Mathematics, University of Florida, Gainesville,
FL 32611, USA}
\email{alladik(at)ufl.edu}
\dedicatory{Dedicated to Jim Lepowsky for his 80th birthday.}
\keywords{Euler’s theorem, Lebesgue’s identity, Capparelli’s partition theorems,
infinite hierarchy, q-hypergeometric identities, method of weighted words.}
\subjclass[2000]{05A15, 05A17, 05A19, 11P81, 11P83}

\begin{document}
\begin{abstract}
 In a recent paper, the authors introduced an infinite hierarchy of $q$-hypergeometric identities, of which the first three orders, $0$, $1$,  and $2$, relate to the partition theorems of Euler, Lebesgue, and Capparelli, and stated a partition theorem at order 4 which lies beyond Capparelli's theorem. Here, we first state certain partition theorems that hold at all even orders beyond Capparelli and provide bijective proofs for these theorems. In doing so, we show that there is a fourfold infinite hierarchy of partition theorems that emanates from Capparelli's theorem, which is the base case. It is also shown that the equality of two of the four generating functions holds for all orders, odd and even. Lastly, a very general framework for the remaining two functions is constructed via the method of weighted words, encompassing all possible orders and yielding several infinite hierarchies with different dilations and translations.
\end{abstract}

\maketitle

\section{Introduction and statement of results}
In a recent paper by the authors, a certain infinite hierarchy of $q$-hypergeometric identities was announced \cite{YAKA25}. To set the stage for the infinite hierarchy, we first state certain classic results that, in addition to providing the historical thread, serve as the initial cases of the hierarchy. 

At the base stands the foundational 18th-century theorem of Euler, namely,
\[
\sum^{\infty}_{j=0}\frac{q^{\frac{j^2+j}{2}}}{(1-q)(1-q^2)\cdots(1-q^j)}
=\prod_{n\geq1}(1+q^{n}).\tag{1.1}
\]

To establish (1.1), one notes that the right-hand side (RHS) plainly counts partitions into distinct parts, and the left-hand side (LHS) does the same because the $j$-th term of the series is the generating function of partitions into exactly $j$ distinct parts.

It is worth noting that (1.1) comes with a companion 
\[\prod_{n\geq1}(1+q^{n})=\prod_{n\geq1}\frac{1-q^{2n}}{1-q^n}=\prod_{n\geq1}\frac{1}{1-q^{2n-1}},\tag{1.2}\]
where the last equality follows because in the middle product, the numerator cancels the corresponding terms in the denominator to yield the product on the right. For the remainder of this paper, we refer to this clever cancellation arising
from the identity $1+x=\frac{1-x^2}{1-x}$, and other similar factorizations, as \textit{Euler's trick}.

For the subsequent level, there is a well-known identity due to V. A. Lebesgue (see \cite{GEA}), namely
\[\sum^{\infty}_{j=0}\frac{q^{\frac{j^2+j}{2}}\prod_{\ell=1}^j(1+aq^{\ell})}{(1-q)(1-q^2)\cdots(1-q^j)}=\prod_{n\geq1}(1+aq^{2n})(1+q^{n}).\tag{1.3}\]
The corresponding partition interpretation of (1.3) is more complicated, as it involves a \textit{weighted partition identity}; the reader is referred to \cite{KABG93} for a detailed treatment. Note that setting $a=0$ in (1.3) yields (1.1). In other words, Euler's identity can be regarded as an instance of Lebesgue's identity.

\begin{remark}

To get regular (unweighted) partition theorems out of Lebesgue’s identity, we transform via  
 \[\text{dilation: }q\to q^2,\]
 \[\text{translations: }a\to aq^{-1}\quad\text{or} \quad  a\to aq^{-3},\] which give rise to the celebrated little G\"ollnitz (see \cite{YAKA25,GEA,HG67})  $q$-hypergeometric identities:

\[
\sum^{\infty}_{j=0}\frac{q^{j^2+j}\prod_{n=1}^j(1+aq^{2n-3})}{(1-q^2)\cdots(1-q^{2j})}
=\prod^{\infty}_{m=1}(1+q^{4m})(1+q^{4m-2})(1+aq^{4m-3}),\tag{1.4a}
\]
and
\[
\sum^{\infty}_{j=0}\frac{q^{j^2+j}\prod_{n=1}^j(1+aq^{2n-1})}{(1-q^2)\cdots(1-q^{2j})}
=\prod^{\infty}_{m=1}(1+q^{4m})(1+q^{4m-2})(1+aq^{4m-1}).\tag{1.4b}
\]
The corresponding partition theorems are well known. We will state the one corresponding to (1.4a) for illustration. The first little G\"ollnitz (see \cite{YAKA25,GEA,HG67}) partition theorem can be phrased as follows:\\

The number of partitions of $n$ into distinct parts with exactly $\omega$ odd parts, each $1\equiv \pmod 4$, is equal to the number of partitions $\pi$ of $n$ in which there are exactly $\omega$ odd parts, such that the gap between parts of $\pi$ is at least $2$, with strict inequality if a part is odd. 
\end{remark}

The next level (Capparelli's partition correspondence) is where things get more intricate, and the character of the hierarchy really begins to take shape. To contextualize, we recall the Rogers--Ramanujan identities.

In the theory of partitions and $q$-hypergeometric series, the Rogers--Ramanujan identities (discovered independently by Rogers and Ramanujan, and first proved by Rogers (see \cite{GEA})), are among the most fundamental: 
\[
RR1:=\sum^{\infty}_{j=0}\frac{q^{j^2}}{(1-q)(1-q^2)\cdots(1-q^j)}
=\prod^{\infty}_{m=1}\frac{1}{(1-q^{5m-1})(1-q^{5m-4})},\tag{1.5a}
\]
and
\[
RR2:=\sum^{\infty}_{j=0}\frac{q^{j^2+j}}{(1-q)(1-q^2)\cdots(1-q^j)}
=\prod^{\infty}_{m=1}\frac{1}{(1-q^{5m-2})(1-q^{5m-3})}.\tag{1.5b}
\]
The combinatorial interpretation of (1.5a) and (1.5b) was discovered independently by MacMahon and Schur and is as follows.\newpage

{\bf{Theorem RR.}} (Rogers--Ramanujan partition theorem)

{\it{For}} $i=1,2$, {\it{the number of partitions of an integer}} $n$ {\it{into
parts that differ by}} $\ge 2$ {\it{and with least part}} $\ge i$, {\it{equals
the number of partitions of}} $n$ {\it{into parts}} $\pm i\pmod{5}$.
\medskip

In general, a Rogers--Ramanujan type identity is one in which a $q$-hypergeometric series (possibly a multiple series) is equal to a product over integers in certain residue classes modulo an integer $M$. The name ``Rogers--Ramanujan type'' derives from (1.5a) and (1.5b), which are the prototypes. Typically, the $q$-hypergeometric series in a Rogers--Ramanujan type identity is the generating function of partitions whose parts satisfy certain difference conditions, while (clearly) the product is the generating function of partitions whose parts satisfy certain congruence conditions. Thus, the combinatorial interpretation of a Rogers--Ramanujan type identity is a Rogers--Ramanujan type partition theorem equating partitions whose parts satisfy certain difference conditions with partitions whose parts satisfy certain congruence conditions.  

Lepowsky and Wilson \cite{LW1,LW2} famously proved the Rogers--Ramanujan identities by a study of vertex operators on Lie algebras. In line with their pioneering work, Capparelli \cite{SC93}, by investigating vertex operators, conjectured\footnote{Andrews was the first to prove Capparelli's conjecture  \cite{GEA94}. Subsequently, Capparelli himself also gave a proof (see \cite{SC96}) of this conjecture using Lie algebras.} the following Rogers--Ramanujan type partition theorems:
\medskip

{\bf{Theorem C.}}
\begin{enumerate}[label=(\roman*)] 
\item \textit{The number of partitions $ \lambda= \lambda_1 +\cdots +\lambda_{\ell}$  of $n$ with every odd part $>1$, $\lambda_i-\lambda_{i+1}\geq 2$, and $\lambda_i-\lambda_{i+1}\geq 4$ whenever $\lambda_i+\lambda_{i+1}\not \equiv 0 \pmod{3}$ is equal to the number of partitions of $n$ into parts congruent to $\pm 2, \pm 3\pmod{12}$.}\\
\item \textit{The number of partitions $ \lambda= \lambda_1 +\cdots +\lambda_{\ell}$  of $n$ with every even part $>2$, $\lambda_i-\lambda_{i+1}\geq 2$, and $\lambda_i-\lambda_{i+1}\geq 4$ whenever $\lambda_i+\lambda_{i+1}\not \equiv 0 \pmod{3}$ is equal to the coefficient of $q^n$ in the product}
\[\prod_{n\geq 0}\frac{(1-q^{12n+2})(1-q^{12n+10})}{(1-q^{12n+1})(1-q^{12n+3})(1-q^{12n+5})(1-q^{12n+7})(1-q^{12n+9})(1-q^{12n+11})}.\]
\end{enumerate}

For the remainder of this paper, we will refer to partitions satisfying the first difference condition in Theorem C as Capparelli partitions. Based on the observation that the number of partitions of $n$ into parts congruent to $\pm 2, \pm 3\pmod{12}$ is equal to the number of partitions of $n$ into distinct parts  $0, 2, 3$ or $ 4 \pmod{6}$, Alladi, Andrews, and Gordon \cite{KAGEABG95} established the following refinement.

\medskip

{\bf{Theorem C-R.}} {\it{Let}} $A_3(n;(i,j),k)$ {\it{denote of partitions of $n$ into distinct parts  $0, 2, 3$ or $ 4$ modulo $6$ with}} $i$ {\it{parts}} $\equiv 4 \pmod{6}$, $j$ {\it{parts}}
$\equiv 2\pmod{6}$, {\it{and}} $k$
{\it{multiples of $3$ that are greater than}} $3(i+j)$. 

{\it{Let}} $D_3(n;(i,j),k)$ {\it{denote the number of Capparelli partitions of $n$ with}}
$i$ {\it{parts}} $\equiv 1\pmod{3}$,\\ $j$ {\it{parts}} $\equiv 2\pmod3$, {\it{and}} $k$ {\it{multiples of $3$}}. {\it{Then}}
\[A_3(n;(i,j),k)=D_3(n;(i,j),k).\]

\medskip

The above refinement and the treatment in \cite{KAGEABG95} are of preeminent importance in this paper. To fully discuss this, it will help to recall some notation. For the remainder of this paper, we denote \[
(A)_n=(A;q)_n=\prod^{n-1}_{j=0}(1-Aq^j),\]
and, by extension, \[(A;q)_{\infty}=\prod^{\infty}_{j=0}(1-Aq^j),\] when $|q|<1$.  Furthermore, we let $T_i={i+1 \choose 2}$ denote the $i$-th triangular number. Lastly, the $q$- multinomial\footnote{In the special case where it is a q-binomial symbol, we may use the notational shorthand ${M+N \brack M}$ or ${M+N \brack N}$ in place of ${M+N \brack M,N}$ as is standard.} symbol is
$${{\nu_1+\cdots+\nu_r}\brack {\nu_1, \cdots, \nu_r}}_{q}=\frac{(q)_{\sum_{i=1}^r\nu_r}}{\prod_{i=1}^r(q)_{\nu_r}}.$$
In general, ${{\nu_1+\cdots+\nu_r}\brack {\nu_1, \cdots, \nu_r}}_{q^2}$ is obtained from ${{\nu_1+\cdots+\nu_r}\brack {\nu_1, \cdots, \nu_r}}_{q}$ by replacing every occurrence of $q$ with $q^2$ and other dilations of $q$ are handled analogously.

The first key element of the treatment of Theorem C-R in \cite{KAGEABG95} is to view it as the partition counterpart of the following $q$-hypergeometric identity\footnote{There are other $q$-hypergeometric identities for Capparelli partitions, such as that of Kur\c{s}ung\"{o}z \cite{KK19}. The double sum expression in \cite{KK19} was found independently by Kanade and Russell \cite{KKMCR19}.}

\[
\sum_{i,j}\frac{a^ib^jq^{2T_i+2T_j}(-q)_{i+j}(-cq^{i+j+1})_{\infty}}{(q^2;q^2)_i(q^2;q^2)_j}=\sum_{i,j,k}\frac{a^ib^jc^kq^{2T_i+2T_j+T_k+(i+j)k}}{(q)_{i+j+k}}{{i+j+k} \brack k}_q
{{i+j} \brack i}_{q^2}.\tag{1.6}
\]

In the case where the combinatorial interpretation of the $q$-hypergeometric identity with free parameters yields a refined partition theorem possibly under dilations and translations, we shall refer to the $q$-hypergeometric identity as the \textit{key identity} for the partition theorem. Thus (1.6) is the key identity
for the refined Theorem C-R as demonstrated in \cite{KAGEABG95}. We will discuss further elements of the treatment in \cite{KAGEABG95} in the sequel, but first, we will show how one may view (1.3) as a logical consequence of (1.6). Begin by setting $c=1$ and $b=0$, then the left-hand side of (1.6) is 
\[
(-q)_{\infty}\sum_{i,j}\frac{a^iq^{2T_i}}{(q^2;q^2)_i}
=(-aq^2;q^2)_{\infty}(-q)_{\infty},
\]
whereas the sum on the right in (1.6) becomes
\[
\sum_{i.k}\frac{a^iq^{2T_i+T_k+ik}}{(q)_i(q)_k}
=\sum_{i,k}\frac{a^iq^{T_i+T_{i+k}}}{(q)_{i+k}}{{i+k}\brack i}_q
=\sum^{\infty}_{n=0}\frac{q^{T_n}}{(q)_n}\sum^n_{i=0}a^iq^{T_i}{n\brack i}_q
=\sum^{\infty}_{n=0}\frac{q^{T_n}(-aq)_n}{(q)_n},
\]
where we have made the change of variable $n=i+k$; thus (1.3) follows from (1.6).

It was with this in mind that K. Alladi in 1994 established and considered (see \cite{YAKA25}) the following infinite hierarchy\footnote{In \cite{ABAKU20,ABAKU22}, Berkovich and Uncu have given a different hierarchy of polynomial identities implying Capparelli's partition correspondence.} of identities for $r=0,1,2,\dots$:
\begin{align*}
\sum_{\nu_1, \nu_2, \cdots, \nu_r, k\ge 0}&\frac{a^{\nu_1}_1a^{\nu_2}_2\cdots a^{\nu_r}_rq^{2T_{\nu_1}+2T_{\nu_2}\cdots +2T_{\nu_r}+T_k+k(\nu_1+\nu_2+\cdots \nu_k)}(-q)_{\nu_1+\nu_2+\cdots\nu_k}}{(q^2;q^2)_{\nu_1}(q^2;q^2)_{\nu_2}\cdots (q^2;q^2)_{\nu_r}(q)_k}\\
&=(-a_1q^2;q^2)_{\infty}(-a_2q^2;q^2)_{\infty}\cdots(-a_rq^2;q^2)_{\infty}(-q)_\infty\tag{1.7}.    
\end{align*}

Although (1.7) provides a beautiful sum-to-product identity, due to technical reasons having to do with the subsequent elements in the treatment in \cite{KAGEABG95}, for the purpose of this paper, we will emphasize the following alternative form of (1.7) with an extra parameter noted in \cite{YAKA25}
\begin{align*}
\sum_{\nu_1, \nu_2, \cdots, \nu_r, k\ge 0}&\frac{a^{\nu_1}_1a^{\nu_2}_2\cdots a^{\nu_r}_rc^kq^{2T_{\nu_1}+2T_{\nu_2}\cdots +2T_{\nu_r}+T_k+k(\nu_1+\nu_2+\cdots \nu_r)}}{(q)_{\nu_1+\nu_2+\cdots\nu_r+k}}{{\nu_1+\cdots+\nu_r}\brack {\nu_1, \cdots, \nu_r}}_{q^2}\\
&\times \quad {{\nu_1+\cdots+\nu_r+k}\brack {\nu_1+\cdots+\nu_r,k}}_{q}\\
&=\sum_{\nu_1, \nu_2, \cdots, \nu_r\ge 0}\frac{a^{\nu_1}_1a^{\nu_2}_2\cdots a^{\nu_r}_rq^{2T_{\nu_1}+2T_{\nu_2}\cdots +2T_{\nu_r}}(-q)_{\nu_1+\nu_2+\cdots\nu_r}(-cq^{\nu_1+\nu_2+\cdots +\nu_r+1})_{\infty}}{(q^2;q^2)_{\nu_1}(q^2;q^2)_{\nu_2}\cdots (q^2;q^2)_{\nu_r}}. \tag{1.8}\\
\end{align*}

In the same paper where the aforementioned hierarchy was announced \cite{YAKA25}, we stated a partition theorem at order $r=4$ beyond Theorem C-R. However, to fully justify referring to (1.7) and (1.8) as key identities for higher variants of Capparelli's partition theorem, we must go beyond $r=4$ and address infinitely many $r$, which was not done previously. To ensure regular (unweighted) partition theorems out of (1.7) and (1.8), we consider the dilations
\[q\to q^\delta,\]
with $\delta\geq r+1$. For the sake of simplicity, we may choose $\delta$ as small as possible. Thus, from this perspective,  $q \to q^{r+1}$ is the \textit{optimal dilation}.\\

The following new result was established by the authors after the submission of \cite{YAKA25}.

\begin{theorem}
For $r$ an even number and $m=r+1$, consider the following four partition functions. Let $A_m(n;\nu_2,\nu_4,\dots\nu_{2m-2},k)$ denote the number of partitions of $n$ into distinct parts, where the only odd parts are multiples of $m$, where the number of multiples of $m$ that are $>m(\nu_2+\nu_4+\dots+\nu_{2m-2})$ is $k$, and for each i, there are $\nu_i$ parts $\equiv i\pmod{2m}$.

Let $B_m(n;\nu_2,\nu_4,\dots\nu_{2m-2},k)$ denote the number of partitions of $n$ into $\nu_{i}$ distinct parts $\equiv i\pmod{m}$, such that the difference between two parts of different parities is $>m$, and $k$ distinct multiples of $m$ each $>m(\nu_2+\nu_4+\dots+\nu_{2m-2})$.

Let $D_m(n;\nu_2,\nu_4,\dots\nu_{2m-2},k)$ denote the number of partitions $ n= \lambda_1 +\cdots +\lambda_{\ell}$ of $n$, where $\lambda_i$ are distinct parts written in decreasing order, where each odd part is $\geq m$, with $k $ multiples of $m$ and $\nu_i$ parts $\equiv i\pmod{m}$ and, such that:\begin{enumerate}[label=(\roman*)]
    \item The gap is $\geq m$ if two consecutive parts have different parities or one of them is a multiple of $m$, with gap $m$ only allowed for consecutive multiples of $m$, and
    \item when $m|\lambda_i$, we have $\lambda_i-\lambda_{i+j}\geq mj$ for $i,j\geq0$ with $i+j\leq \ell$.
\end{enumerate}
\medskip
Let $D'_m(n;\nu_2,\nu_4,\dots\nu_{2m-2},k)$ denote the number of partitions $ n= \lambda_1 +\cdots +\lambda_{\ell}$ of $n$, where $\lambda_i$ are distinct parts written in decreasing order, where each odd part is $\geq m$,  with $k $ multiples of $m$ and $\nu_i$ parts $\equiv i\pmod{m}$ and, such that:\begin{enumerate}[label=(\roman*)]
    \item The gap is $\geq m$ if two consecutive parts have different parities or one of them is a multiple of $m$, with gap $m$ only allowed for consecutive multiples of $m$, and
    \item when $m|\lambda_{i+j}$, we have $\lambda_i-\lambda_{i+j}\geq mj$ for $i,j\geq0$ with $i+j\leq \ell$.
    \item If $m|\lambda_i$, then $\lambda_i>m(\ell-i)$.
\end{enumerate}
\medskip
Then 
\[A_m(n;\nu_2,\dots\nu_{2m-2},k)=B_m(n;\nu_2,\dots\nu_{2m-2},k)=D_m(n;\nu_2,\dots\nu_{2m-2},k)=D'_m(n;\nu_2,\dots\nu_{2m-2},k).\]
\end{theorem}

\begin{remark} When $r=2$, conditions (i) and (ii) are redundant because they always hold. All this will be clear from the combinatorial (bijective) proof of Theorem 1, which we provide in the next section.\end{remark}

\begin{remark}
It is to be noted that in identity (1.7), which gives the infinite hierarchy, we always have a product representation for the generating function for all $r$. Indeed, this was a major motivation for introducing this infinite hierarchy. However, under the dilation $q\mapsto q^{r+1}$ with $r+1=m$, from identity (1.7), we are unable to keep track of the number of parts which are multiples of $m$. It is for this reason that the parameter $c$ has been included in the more general identity (1.8), even though we do not get a product representation. After the optimal dilation, the power of $c$ on the left-hand side of (1.8) keeps track of the number of large multiples of $m$, whereas on the right-hand side of (1.8), the power of $c$ represents the number of multiples of $m$ in the partition whose parts satisfy generalized Capparelli difference conditions (see Theorem 1). When $c=1$, identity (1.8) reduces to (1.7), which is when we have a product representation.
\end{remark}

The second element of the treatment from \cite{KAGEABG95} that is extended here is the use of a certain seven-step process bijectively proving Theorem C-R. This variant is presented in the next section in a manner explicitly illustrating \cite[Thm.~C5]{YAKA25}.

\begin{remark}It is to be noted that while we follow the combinatorial method presented in \cite{KAGEABG95} to prove these theorems for even $r\ge 4$, there is an essential difference in the proof which leads to a {\it{dichotomy}}, and this is what generates two partition theorems at each $r\ge 4$. Such a dichotomy does not occur at $r=2$. Specifically, the difference between $D_m$ and $D'_m$ is in step 6. The first inserts a multiple of $m$ to the highest position, such that no part below is larger; The second inserts a multiple of $m$ to the lowest position, such that no part above it is smaller. Then, in the second case, for example, when you go to step $7$, each one of those parts that was already bigger is going to get a $+m$ for each part between it and the closest $m$ below it (including one endpoint), which gives condition (ii). However, we note that while this multifurcation of hierarchies is a new phenomenon, it maintains characteristics stemming from the classic Capparelli case described by Theorem C-R (which corresponds to $r=2$).\end{remark}

The third element of the treatment from \cite{KAGEABG95} that is extended here is the generalization to the weighted words setting (see also \cite{KABG93-2}).  What is interesting here is that once the work is done in the weighted words framework, we can view relations in Theorem 1 as emerging from a special instance of a more general result under certain dilations and translations. Furthermore, the weighted words treatment works for all $m$, not just odd $m$. With the general weighted words partition theorem, we will see that we may obtain different identities under certain dilations and translations of the weighted words counterparts. We will mention two instances of this phenomenon.

First, recall that in fact Capparelli stated two conjectures for $r=2$. Note that it is easy to see that an application of the Euler trick on (1.6) gives 
\[\prod_{n\geq0}(1+q^{3n+3})(1+q^{6n+1})(1+q^{6n+5}).\]
The interpretation of the above is evidently the generating function for distinct parts $\equiv 1,3,5, \,\text{or}\,\,\, 6\pmod 6$. Actually, as is already noted in \cite{KAGEABG95}, under certain dilations and translations, one can obtain as a consequence of \cite[Thm.~3]{KAGEABG95} the following.\\

\textbf{Theorem C*-R.}{\it{ Let}} $A_3^*(n;i,j,k)$ {\it{denote the number of partitions of $n$ into distinct parts  $0, 1, 3$ or $ 5$ modulo $6$ with}} $i$ {\it{parts}} $\equiv 1\pmod{6}$, $j$ {\it{parts}} $\equiv 5 \pmod{6}$, {\it{and}}  $k$ {\it{multiples of $3$ that are greater than}} $3(i+j)$.

{\it{Let}} $D_3^*(n;i,j,k)$ {\it{be the number of partitions of}} $n$ {\it{into distinct parts of the type enumerated by}} $D_3(n;i,j,k)$ {\it{with the only difference being that instead of 1 being disallowed as a part, we now have 2 disallowed as a part. Then}}
$$
A_3^*(n;i,j,k)=D_3^*(n;i,j,k).
$$

\bigskip

We will state a second instance of variants arising from more general choices of translation and dilations. However,  we must recall some conventions stated in \cite{YAKA25}. We write $p\equiv j\pmod{M}$ to mean that $p$ is of the form $j+\lambda M$ with $\lambda\ge 0$ an integer.  In the case $M\!\not|\,j$, we say that  $p=j+\lambda M$ has odd (resp. even) level parity if  $\lambda$ is odd (resp. even).

We are now ready to state the second instance of such variants. Namely, the following is a refinement of Theorem H from \cite{YAKA25}.\\

{\bf{Theorem H-R. }}{\it{Given a modulus}} $2m$, {\it{and a positive integer}}  $r<m$, {\it{consider}} $r$ {\it{distinct integers}} $j_i\in (0,2m)$, {\it{pairwise incongruent modulo $m$, with}} $i=1,2,\cdots r$, {\it{and no}} $j_i$ {\it{equal to}} $m$.

{\it{Let}} $A(n;\nu_1, \nu_2, \cdots, \nu_r, k; 2m)$ {\it{denote the number of partitions of}} $n$ {\it{into}} $\nu_i$ {\it{distinct parts}}
$\equiv j_i\pmod{2m}$ {\it{for}} $i=1,2,\cdots, r.$ {\it{and distinct parts}} $\equiv 0\pmod{m}$, {\it{of which exactly}} $k$ are $>m(\nu_1+\nu_2+\cdots+\nu_r$).

{\it{Let}} $B(n; \nu_1,\nu_2,\cdots,\nu_r, k; m)$ {\it{denote the number of partitions of}} $n$ {\it{into}} $\nu_i$ {\it{distinct parts}} $\equiv j_i\pmod{m}$
{\it{for}} $i=1,2,\cdots, r,$ {\it{such that the difference between two parts of different level parities is}} $>m$, {\it{and distinct parts}} $\equiv 0\pmod{m}$ {\it{of which there are exactly}} $k$ {\it{and each}} $>(\nu_1+\nu_2+\cdots+\nu_r)m$.

{\it{Then, we have}}
\[
A(n;\nu_1, \nu_2, \cdots, \nu_r,k; 2m)=B(n; \nu_1,\nu_2,\cdots,\nu_r,k;m).
\]

\bigskip
A combinatorial proof of Theorem H-R is, in essence, the first three steps of the more general Theorem 1,  which is proved in the next section, and Theorem 2, which is stated and proved in the section after.

Observe that, in infinitely many instances, Theorem H-R is simply one of the three equalities in Theorem 1 and hence a special case. However, one of the reasons for emphasizing Theorem H-R is that it works for all $m$, not just for odd $m$. This is not to say that one can't extend Theorem 1 to all $m$; the issue is that there are always difficulties with the function $D_m$ when $m$ is even, as is exhibited by the situation with the Lebesgue identity corresponding to a weighted partition identity, which is discussed in length in \cite{KABG93}. 

The reader may wonder, if Theorem 1 is simply a special case of the more general weighted words treatment, then why is it emphasized?  There are at least two reasons for this. First, as the reader will see, the general statement is quite abstract, and we feel that a concrete counterpart will be appreciated. Second, the particular concrete statement is given yet another companion. More precisely, in our case, where $r+1=m$ is odd, and the translations are those appropriate for Theorem 1, we have
\[
(-q^m;q^m)_{\infty}\prod^{m-1}_{j=1}(-q^{2j};q^{2m})_{\infty}
=\frac{1}{(q^m;q^{2m})_{\infty}\prod^{m-1}_{j=1}(q^{4j-2};q^{4m})_{\infty}}.\tag{1.9}
\]
Identity (1.9) is established using Euler's trick, namely, replacing every
term of the form $(1+x)$ on the left-hand side of (1.9) by $(1-x^2)/(1-x)$ and
performing the cancellations. \\ 

In view of the last identity, we may quickly record an immediate corollary.
\begin{corollary}
For odd $m$, let $C_m(n)$ denote the number of partitions of $n$ where every part is either congruent to $2$ modulo $4$ or an odd multiple of $m$. Then 
\[A_m(n)=B_m(n)=C_m(n)=D_m(n)=D'_m(n).\]
\end{corollary}
It is to be noted that the generating function for $C_m(n)$ is a product.

The last element of the treatment in \cite{KAGEABG95}, which justifies calling (1.6) a key identity and will be given an analogy to for the general case of an arbitrary but finite number of colors, is the treatment of local generating functions directly emerging from the key identities. In \cite{KAGEABG95}, it was demonstrated that \[H(i,j,k)=q^{2T_i+2T_j+T_k+(i+j)k}{{i+j+k} \brack k}_q{{i+j} \brack i}_{q^2},\]
counts minimal colored partitions of the Capparelli type. This is one of the key factors justifying the claim that (1.6) is the key identity of Theorem C-R.  From this view, it is natural to ask (see \cite[Rmk.~8.7]{YAKA25}) if the analogous function 
\[
H(\nu_1,\dots,\nu_r,k)=q^{2T_{\nu_1}+\cdots +2T_{\nu_r}+T_k+k(\nu_1+\nu_2+\cdots \nu_r)}{{\nu_1+\cdots+\nu_r}\brack {\nu_1, \cdots, \nu_r}}_{q^2}\times{{\nu_1+\cdots+\nu_r+k}\brack {\nu_1+\cdots+\nu_r,k}}_{q}\]
plays a similar role. Unfortunately, if the role is to precisely count minimal partitions, we will see that it fails even in the immediately next highlighted case $r+1=m=5$. However, we will show that it counts what we call \textit{mock-minimal partitions}.\footnote{The adjective ``mock" is used here only for a specific set of partitions. The authors do not establish a relation to the more common use of the term in the theory of $q$-series in this manuscript.} In doing this, we will give new insight into even the classic case, going beyond what was done in \cite{KAGEABG95}.

The paper’s structure generally follows the chronology of the goals set in this introduction with one exception. It is not enough to merely state (1.7) and (1.8) and call them ``key identities" for the infinite partition hierarchy introduced here and its weighted words counterpart, but in fact to justify why these identities fit that title. On the one hand, the reader will already see that the RHS of (1.8) plainly counts the weighted words generalization of $A_m(n;\nu_2,\dots,k)$ and the title becomes justified once Theorem 4 is proven, as it gives that the LHS of (1.8) counts the weighted words generalizations of $D_m(n;\nu_2,\dots,k)$ and $D'_m(n;\nu_2,\dots,k)$. Thus, the aforementioned goal of justifying calling (1.7) and (1.8) the relevant ``key identities" can only be achieved at the very end. Nonetheless, the second section achieves the goal of giving a bijective proof of Theorem 1 and Theorem H-R. Likewise, the third section presents the weighted words treatment. Lastly, the final section treats mock-minimal partitions and mock-local generating functions.

\section{The bijective proof of Theorem 1 and Theorem H-R}
The bijection is a variant of the seven-step process\footnote{The authors of \cite{KAGEABG95} based aspects of their bijection on previous ideas of Bressoud \cite{DMB78,DMB79}.} from \cite{KAGEABG95}, adapted to our setting.\\

In what follows, we identify the partitions counted by $A_m$ with bipartitions
$(\pi_1; \pi_2)$, where $\pi_1$ is composed of distinct even non-multiples of $m$, and $\pi_2$ is composed of distinct multiples of $m$. We now give the seven-step process and explicitly illustrate each step starting from the bipartition
$(\pi_1; \pi_2)$ where\footnote{The example used in the proof is based on earlier unpublished notes by the authors.}

\[
\pi_1=28+26+14+12+6+4+2
,\]
and
\[
\pi_2=55+45+40+20+10.
\]

\medskip

\textbf{Step 1:} Denote the number of parts of $\pi_1$ by $\nu$ and separate $\pi_2$ into a partition $\pi_4$ composed of only the parts that are $\le m\nu$, and a partition $\pi_5$ having the parts $>m\nu$.
\[
\pi_4=20+10, \quad \pi_5=55+45+40. \quad (m=5)
\]

\medskip

\textbf{Step 2:} Write $\pi_4$ as an $m$-modular Ferrers graph. Then conjugate that graph to obtain $\pi_4^*$.
\[\pi^*_4=10+10+5+5\]

\medskip

\textbf{Step 3:} Construct $\pi_6=\pi_1+\pi^*_4$. This is the sum of the two partitions obtained by adding the largest element of $\pi_1$ to the largest element of $\pi_4$, the second largest to the second largest, and so on.  

\[
\pi_6=38+36+19+17+6+4+2.
 \quad (m=5)\]

\medskip

In what follows we show that the correspondence $\pi_6 \leftrightarrow (\pi_1, \pi^*_4)$ is bijective.\\

{\it{Reversibility of Step 3:}} Exactly as in \cite{KAGEABG95}, the reversibility in Step 3 is seen by noting that parity switches occur exactly at the positions where the insertions occur. More precisely, to get partitions $\pi_1$ and $\pi^*_4$ out of $\pi_6=h_1+h_2+\cdots+h_{\nu}$, start from the smallest part of $\pi_6$ and move upwards. Mark the position where the first odd part, say, $h_{j_r}$ is
encountered (if such exists. If no odd part is encountered then $\pi^*_4$ is a null partition). If $h_{j_r}$ exists,  then this corresponds to the largest part $mj_r$ of $\pi^*_4$, written as a column of length $j_r$ with each part in the column being equal to $m$. Next, moving upward beyond $h_{j_r}$, mark the position of the first part where the parity is even (if one exists), say $h_{j_{r-1}}$. This corresponds to the second largest part $mj_{r-1}$ of $\pi^*_4$, written as a column of length $j_{r-1}$ with each part equal to $m$. Proceed in this fashion, keeping track of parity changes until the last parity change occurs at position $j_1$ with part $h_{j_1}$ of $\pi_6$. This corresponds to the smallest part $mj_1$ of $\pi^*_4$. Having marked the positions where these parity changes occur, we may peel off (extract) $\pi^*_4$ out of $\pi_6$, to yield $\pi_1$ after the extraction. Thus, Step 3 is a bijection.\\

\textit{Completion of the proof of Theorem H-R}: The bipartition $(\pi_6, \pi_5)$ may be viewed as a regular partition, since the parts of $\pi_6$ are non-multiples of $m$ and the parts of $\pi_5$ are
multiples of $m$ each $>m\nu$. Note that the parts of $\pi_6$ are congruent to the corresponding parts of
$\pi_1$ modulo $m$. Also note that since odd parts of $\pi_6$ are created by embedding
$\pi^*_4$ into $\pi_1$, all odd parts of $\pi_6$ are $>m$. Thus $\pi_6$ is the partition enumerated by the function $B$ in Theorem H-R and the function $B_m$ in Theorem 1, and the combinatorial construction described above, which is reversible, proves  Theorem H-R and the first equality in Theorem 1.

\medskip

\textbf{Step 4:} Arrange $\pi_5$ in a vertical column in decreasing order; then append $\pi_6$ underneath it in the same vertical column and also in decreasing order (as in Table 1). This is denoted by $\pi_5/\pi_6$ here and throughout the paper.

$$
\pi_5/\pi_6: 55+45+40+38+36+19+17+6+4+2 \quad (m=5)
$$
\medskip

\textbf{Step 5:} Working from the bottom-most part and systematically moving upwards, subtract $0,m,2m, \dots$ from successive elements of the $\pi_5/\pi_6$ column. Write the result in a new column (labeled $C_1$) side by side with another new column (labeled $C_2$) featuring the subtracted integers.
$$
C_1: (10, 5,5, 8, 11, -1, 2, -4, -1, 2)
$$
$$
C_2: (45, 40, 35, 30, 25, 20, 15, 10, 5, 0)
$$

\textbf{Note 3.1:} The elements of $C_1$ could be negative, as in the above example. But this does not matter because the multiples of $m$ that will be added to these negative integers (if they exist) will be at least as large as the multiples of $m$ that were subtracted (see Remark after Step 7) to get these negative values. Such negative values DO NOT show up in the proof of Capparelli's theorem. This is due to the following reason: In the case of Capparelli's first theorem, we have
in $\pi_1$, distinct parts $2,4\pmod{6}$. Thus, the parts of $\pi_1$ are taken from the integers $(2,4), (8,10), (14,16), \cdots.$. {\it{Note that the gap between consecutive pairs of integers is 6}}. Insertion of multiples of 3 from $\pi^*_4$ will only increase the gaps. In any case, the tightest packing of parts in $\pi_6$ will be $2+4+8+10+14+16+\cdots$. From these parts, in Step 5, consecutive multiples of $3$ starting from $0$ are subtracted: $0, 3, 6, 9, \cdots$. Note that subtracting $(0,3)$ from $(2,4)$ does not yield any negative integers, and moving upward, we do not get any negative numbers because of the gap 6 property noted above. Thus, in the proof of Capparelli's first theorem (as given in \cite{KAGEABG95}, no negative values occur in $C_1$ in Step 5. A similar reasoning tells us that no negative values will occur in the proof of Capparelli's Second Theorem because even in the tightest packing $(1,5), (7,11), (13,17), \dots$, subtraction of $(0,3), (6,9), (12,15), \dots$, does not yield negative values. In the case of odd $m \ge 5$ considered above, we have more residue classes $\pmod{2m}$, such as $2,4,6,8 \pmod{10}$. So when consecutive multiples of $m$ starting from $0$ are subtracted, one will get negative integers if we start with the tightest packing; for example, if $(0,5, 10, 15)$ are subtracted from $(2,4,6,8)$. This is an essential difference between the Capparelli case $m=3$ and the higher Capparelli cases $m>3$.

\medskip

The remaining two steps depend on a choice between one of two insertion methods. Each gives a
valid partition theorem.\medskip

\textbf{Step 6d:} Rearrange $C_1$ by inserting ONLY the multiples of $m$ in non-decreasing order in {\it{maximal}} positions such that, each multiple of $m$ is larger than
all the non-multiples of $m$ below; since these insertions are at maximal positions, it follows that the first non-multiple of $m$ above a multiple of $m$ must be
larger than that multiple of m. Denote this rearranged column by $C^{r_d}_1$, and display $C^{r_d}_1$ and $C_2$ side by side: $C^{r_d}_1|C_2$. 

$$
C^{r_d}_1: (8, 11, 10, 5, 5, -1, 2, -4, -1, 2)
$$
$$
C_2: (45, 40, 35, 30, 25, 20, 14, 10, 5, 0)
$$

\textbf{Note 3.2:} After the rearrangement, while it is
guaranteed that the first non-multiple of m
(call it $b$) above  a multiple of $m$ (call it $c$), 
satisfies $b>c$, the non-multiple of $m$ above 
$b$ in $C^{r_d}_1$, need not be larger than $c$. 
In the above example,  $m=5$ and $11$ is greater than $10$, but $8$ is not.

\medskip

\textbf{Step 6u:} Rearrange $C_1$ by inserting ONLY the multiples of $m$ in non-decreasing order in {\it{minimal}} positions such that, each multiple of $m$ is smaller than
all the non-multiples of $m$ above; since these insertions are at minimal positions, it follows that the first non-multiple of $m$ below a multiple of $m$ must be
smaller than that multiple of $m$. Denote this rearranged column by $C^{r_u}_1$, and display $C^{r_u}_1$ and $C_2$ side by side: $C^{r_u}_1|C_2$. 

$$
C^{r_u}_1: ( 10,8, 11, 5, 5, -1, 2, -4, -1, 2)
$$
$$
C_2: (45, 40, 35, 30, 25, 20, 14, 10, 5, 0)
$$

\medskip

{\it{Reversal of Step 6}}: Clearly both Steps 6u and 6d are reversible, because to go from Columns $C^{r_u}_1$ and $C^{r_d}_1$ in Steps 6u and 6d to Column $C_1$ in Step 5, take all the multiples of $m$ in $C^{r_u}_1$ and $C^{r_d}_1$ and place them at the top in descending order, and below then place the non-multiples of $m$ in descending order to get $C_1$.\\

{\it{A special feature of Capparelli's theorem:}} In the case $m=3$ which corresponds to Capparelli's Theorem C, Steps 6u and 6d are identical. That is, whenever a multiple of $3$, say $3j$, is inserted such that every non-multiple of $3$ below
$3j$ is smaller, then automatically, every non-multiple of 3 sitting above $3j$ in $C^R_1$ is larger. Thus, this bifurcation into Steps 6u and 6d occurs only when $m>3$. This is to say that, starting from a single Capparelli Theorem C, we get two infinite hierarchies of theorems given by Steps 6u and 6d.\\

\textbf{Step 7d:} Add the corresponding elements of $C^{r_d}_1$ and $C_2$ to get a partition
$\pi_3$ enumerated by $D(n)$.

$$
\pi_3= 53+51+45+35+30+19+17+6+4+2
$$
\medskip

\textbf{Step 7u:} Add the corresponding elements of $C^{r_u}_1$ and $C_2$ to get a partition
$\pi_3$ enumerated by $D'(n)$.

$$
\pi_3= 55+48+46+35+30+19+17+6+4+2
$$
\medskip

{\it{Reversal of Step 7:}} Clearly Steps 7u and 7d are reversible, because given $\pi_3$, subtract $0, m, 2m, \cdots$ in succession starting from the smallest part of $\pi_3$ to get $C^R_1$ in Step 6u and 6d respectively.\\

\begin{center}
    
\begin{table}[ht]
\caption{Steps 4 to 7d/7u of the bijection for Theorem 1 (regular integers)}
\begin{tabular}{cc|cc|cc|ccc}
    Step 4 & \multicolumn{2}{c}{Step 5} & \multicolumn{2}{c}{Step 6d}& \multicolumn{2}{c}{Step 6u}& Step 7d&Step 7u\\
    $\pi_5/\pi_6$ & \multicolumn{2}{c}{$C_1|C_2$} & \multicolumn{2}{c}{$C_1^{r_d}|C_2$} & \multicolumn{2}{c}{$C_1^{r_u}|C_2$}& $\pi_3$& $\pi_3$\\
    \hline
    \\
      55& \quad 10 & 45 \quad\quad & 8 & 45&10&45&53&55\\
      45& \quad 5 & 40 \quad\quad & 11 & 40&8&40&51&48\\
      40& \quad 5 & 35 \quad\quad & 10 & 35&11&35&45&46\\
      38& \quad 8 & 30 \quad\quad & 5 & 30& 5 & 30&35&35\\
      36& \quad 11 & 25 \quad\quad & 5 & 25&5&25&30&30\\
      19& \quad -1 & 20 \quad\quad & -1 & 20&-1&20&19&19\\
      17& \quad 2 & 15 \quad\quad & 2 & 15&2&15&17&17\\
      6& \quad -4 & 10 \quad\quad & -4 & 10&-4 & 10&6&6\\
      4& \quad -1 & 5 \quad\quad & -1 & 5&-1 & 5&4&4\\
      2& \quad 2 & 0 \quad\quad & 2 & 0 &2&0&2&2\\
\end{tabular}
\end{table}
\end{center}

Thus, we have seen that every one of these steps is reversible, and this completes the bijective proof of Theorem 1.\qed\\

\begin{remark} In the rearrangement, due to the insertion of the multiples of m, each non-multiple of $m$ in $C_1$ will stay put or will move up. So even if this non-multiple of $m$ is negative in $C_1$ owing to a subtraction of the corresponding multiple if $m$ in $C_2$, in Step 7, either that multiple of $m$ or a larger multiple of $m$ from $C_2$ will be added to this negative value, and so it will end up in $\pi_3$ as a positive integer (part of a partition).\end{remark}

\bigskip

\textbf{Note 3.3:} Condition (ii) in the definition of $D_m(n)$ and conditions (ii-iii) in $D'_m(n)$ are a consequence of their respective last two steps.
\medskip

\textbf{Note 3.4:} In the case of the combinatorial proof of refined Capparelli
Theorem C-R  in \cite{KAGEABG95}, since $r=2, m=3$, it is possible
to insert the multiples of 3 (say $3\lambda$) in Step 6 such that all multiples of $3$ lying below $3\lambda$ are less than $3\lambda$, and ALL non-multiples of
$3$ that lie above $3\lambda $ are greater than $\lambda $. So in the
Capparelli case, $D=D'$. It is only when $r\ge 4$ that we get two different but
equal-valued partition functions, and thus two hierarchies of partition functions emanating from the refined Capparelli Theorem C-R.\\ 

\textbf{Note 3.5:} Just as we have infinite hierarchies of partition theorems emerging from Theorem C-R, we will also have analogous infinite hierarchies of partition theorems emerging from Theorem C*-R. We do not state these here; instead, we provide a method of a weighted words approach in subsequent sections, which provides partition theorems at all orders $r\ge 3$ odd and even, and depending on the dilations and translations used, will yield the corresponding infinite hierarchy of partition theorems.

\section{A bijective proof of the weighted words generalization of Theorem 1}

Here, we consider an ordering scheme where there is one $c$-color, an infinite collection of $a$-type colors $a_i$ for each integer $i>0$,\footnote{In the sequel, the phrases $\lambda$ is an $a$-part or $b$-part mean $\lambda$ has color $a$-type or $b$-type. In more specific instances, the phrases $\lambda$ is an $a_i$-part or $b_j$-part for some $i,j>0$ will be used. Thus, both $2_{a_1}$ and $2_{a_2}$ are $a$-part but only the first is an $a_1$-part. In particular $2_{a_2}$ is not an $a_1$-part.} and an infinite collection of $b$-type colors $b_j$ for each integer $j>0$, ordered as follows:
$$1_{a_1}<2_{b_1}\dots<1_{a_{x}}<2_{b_x}<\dots<1_c<2_{a_1}<3_{b_1}<\dots2_{a_y}<3_{b_y}<\dots<2_c<\dots$$

Note that in the above, there are infinitely many colors less than any given $c$-integer and infinitely many between any two given integers of the same color. Notice also that to each $c$-part there is a successor part but no predecessor (i.e, a $c$ part is not a successor part of any other part).

The above order can be axiomatically defined by the usual axioms of an order, along with the following. 
\begin{itemize}
    \item $n_x\not<n_x$
    \item $n_x<(n+1)_x$
    \item $n_x<n_y$ if neither is a $c$-part and both have the same color type with $y$ having a bigger index.
    \item $n_x<(n+1)_y$ if $x=a_i$ and $y=b_j$ with $j\geq i$ or $x=c$ and y is any $a$-type part. 
    \item $(n+1)_x<n_y$ if $x$ is a $b$-type and $y$ is an $a$-type with \textbf{strictly higher index} or $y=c$ 
    
\end{itemize}

For us, we will only be concerned about the case where only finitely many colors appear. Specifically, we look at a situation where there is one $c$-color, $r'$ $a$-colors $a_1,\dots, a_{r'}$, and $r''$ $b$-colors $b_1,\dots,b_{r''}$, with $r'+r''=r$ and $r+1=m$. Note that it is not actually simpler to restrict to an arbitrary finite order, since if $r'\not=r$ then you will get a strange segment at the end of each level having one color. For example if $r'=2$ and $r''=5$ this looks like
$$1_{a_1}<2_{b_1}<1_{a_{2}}<2_{b_2}\underbrace{<2_{b_3}<2_{b_4}<2_{b_5}}<1_c<2_{a_1}<3_{b_1}<\dots$$

With this set up in mind, let $K(n;\nu_{a_1},\dots,\nu_{a_{r'}},\nu_{b_1},\dots,\nu_{b_{r''}},k)$ denote the number of partition vectors of the form $(\pi_{a_1},\dots,\pi_{a_{r'}},\pi_{b_1},\dots,\pi_{b_{r''}},\pi_c)$ where $\pi_c$ has $k$ distinct parts each of color $c$, and, for $x\not=c$, $\pi_x$ has $\nu_x$ distinct even parts each of color $x$.\\

Let $G(n;\nu_{a_1},\dots,\nu_{a_{r'}},\nu_{b_1},\dots,\nu_{b_{r''}},k)$ denote the number of partitions of $n$ into colored integers, where the color $x\not=c$ occurs $\nu_x$ times and $c$-parts of size greater than the number of non-$c$-parts occur $k$ times, satisfying the difference conditions below:
\begin{enumerate}[label=(\Roman*),ref=\Roman*]
    \item Difference conditions on the neighbors of $c$-parts. That is
    \begin{enumerate}[label=(\theenumi.\roman*),ref=\theenumi.\roman*]
        \item If $n_c>m_x$ are consecutive parts then $(n-1)_c\geq m_x$.
        \item If $m_x>n_c$ are consecutive parts then $(m-1)_x\geq n_c$.
    \end{enumerate}
    \item Difference conditions on the non-$c$-parts
    \begin{enumerate}[label=(\theenumi.\roman*),ref=\theenumi.\roman*]
        \item If $(n+1)_{a_i}>n_{a_j}$ are consecutive parts, then $i>j$.
        \item If $(n+1)_{b_i}>n_{b_j}$ are consecutive parts, then $i>j$.
        \item If for some $j$, $n_{a_j}$ occurs as a part, then for no $i$ can $(n+1)_{b_i}$ occur as a part.
    \end{enumerate}
    
    \item The downwards difference condition: If $m_x<n_c$ occur then $m_x\leq(n-d)_c$ where $d=|\{\lambda\in \pi :<x\leq n_c\}|$.

\end{enumerate}

Analogously, let $G'(n;\nu_{a_1},\dots,\nu_{a_{r'}},\nu_{b_1},\dots,\nu_{b_{r''}},k)$ denote the number of partitions of $n$ into colored integers, where the color $x\not=c$ occurs $\nu_x$ times and $c$-parts of size greater than the number of non-$c$-parts occur $k$ times, satisfying the both (I) and (II) but with an \textit{upwards} condition instead of (III) along with a \textit{supplementary condition} as described below:
\begin{enumerate}[label=(\Roman*),ref=\Roman*]
    \item Difference conditions on the neighbors of $c$-parts. That is
    \begin{enumerate}[label=(\theenumi.\roman*),ref=\theenumi.\roman*]
        \item If $n_c>m_x$ are consecutive parts then $(n-1)_c\geq m_x$.
        \item If $m_x>n_c$ are consecutive parts then $(m-1)_x\geq n_c$.
    \end{enumerate}
    \item Difference conditions on the non-$c$-parts
    \begin{enumerate}[label=(\theenumi.\roman*),ref=\theenumi.\roman*]
        \item If $(n+1)_{a_i}>n_{a_j}$ are consecutive parts, then $i>j$.
        \item If $(n+1)_{b_i}>n_{b_j}$ are consecutive parts, then $i>j$.
        \item If for some $j$, $n_{a_j}$ occurs as a part, then for no $i$ can $(n+1)_{b_i}$ occur as a part.
    \end{enumerate}
    
    \item The upwards difference condition: If $n_c<m_x$ occur then $n_c\leq(m-d)_x$ where $d=|\{\lambda\in \pi :n_c<\lambda\leq m_x\}|$.

    \item The supplementary condition: If $n_c \in \pi$ then $\{\lambda\in\pi:\lambda<n_c\}<n$.

\end{enumerate}

We have the following generalization of Theorem 1.

\begin{theorem}
\[K(n;\nu_{a_1},\dots\nu_{a_{r'}},\nu_{b_1},\dots\nu_{b_{r''}},k)=G(n;\nu_{a_1},\dots\nu_{a_{r'}},\nu_{b_1},\dots\nu_{b_{r''}},k)=G'(n;\nu_{a_1},\dots\nu_{a_{r'}},\nu_{b_1},\dots\nu_{b_{r''}},k).\]
\end{theorem}
We now demonstrate that Theorem 1 follows from Theorem 2.
\begin{corollary}
    Theorem 1 holds.   
\end{corollary}
\begin{proof}
For odd $m$ with $r'=r''=\frac{r}{2}$, the standard transformations sending 
$n_{a_i}\to mn-2(r'+1-i)$, $n_{b_i}\to mn-2(r+1-i)$ and $n_c\to mn$, recover the partition theorem for regular integers.
\end{proof}

Even for Theorem 2, the bijection is a variant of the seven-step process from \cite{KAGEABG95} similar to the proof of \cite[Thm.~C5]{YAKA25} presented in the previous section. This is done below.\footnote{As the presentation for the last section was very detailed, and the key ideas are analogous, the presentation for the weighted word variant will be more streamlined.}\\

\begin{proof}

Start with $\pi_{ab}=2_{b_3}+2_{b_4}+4_{b_2}+10_{a_2}+18_{b_1}+22_{a_1}+30_{b_2}$ and $\pi_c=1_c+4_c+7_c+8_c+17_c$.\\

\textbf{Step 1:} Separate $\pi_c$ into a partition composed of only the parts of size at most equal to the number of non-$c$-parts, which is $1_c+4_c+7_c$, and one composed of the remaining parts, namely $\hat{\pi}_c=9_c+17_c$.\\

\textbf{Step 2:} Forget the color of $1_c+4_c+7_c$, and conjugate to get $1+1+1+2+2+2+3$.\\

\textbf{Step 3:} Embed the result of the previous step to $\pi_{ab}$ to get $\hat{\pi}_{ab}=3_{b_3}+3_{b_4}+5_{b_2}+12_{a_2}+20_{b_1}+24_{a_1}+33_{b_2}$.\\

The example partiticolored the remaining steps will be illustrated in Table 2.\\ 

\textbf{Step 4:}
Write $\hat{\pi}_{c}$ in a vertical column in descending order, and below it write $\hat{\pi}_{ab}$ in the same vertical column and also in descending order (as in  Table 2).\\

\textbf{Step 5:}
Subtract $0$ from the bottom-most part, $1$ from the second bottom-most part, $2$ from the third smallest part and so on. Write the result in a column (labeled $C_1$) side by side with another column (labeled $C_2$) featuring the subtracted integers.\\

The remaining two steps depend on a choice between one of two insertion methods. Each gives a valid partition theorem.\\

\textbf{Step 6u:}
Rearrange the parts by \textit{only} inserting each $c$-part in the lowest position such that it is smaller than every non-$c$-part above it.\\

\textbf{Step 6d:}
Rearrange the parts by \textit{only} inserting each $c$-part in the highest position such that it is greater than every non-$c$-part below it.\\

\textbf{Step 7:}
Add back the members of $C_1$ to the rearranged column you obtained from the previous step.

\begin{center}
    
\begin{table}[ht]
\caption{Steps 4 to 7u/7d of the bijection for Theorem 2 (colored integers)}
\begin{tabular}{cc|cc|cc|ccc}
    Step 4 & \multicolumn{2}{c}{Step 5} & \multicolumn{2}{c}{Step 6u}& \multicolumn{2}{c}{Step 6d}& Step 7u&Step 7d\\
    $\hat{\pi}_c/\hat{\pi}_{ab}$ & \multicolumn{2}{c}{$C_1|C_2$} & \multicolumn{2}{c}{$C_1^{r_u}|C_2$} & \multicolumn{2}{c}{$C_1^{r_d}|C_2$}& $\pi_3$& $\pi_3$\\
    \hline
    \\
      $17_{c}$& $9_c$&$8$& $27_{b_2} $&$ 8$ &$27_{b_2}$&$8$&$35_{b_2}$&$35_{b_2}$\\
      $8_{c}$& $1_c$&$7$& $19_{a_1} $&$ 7$ &$19_{a_1}$&$7$&$26_{a_1}$&$26_{a_1}$\\
      $33_{b_2}$& $27_{b_2}$&$6$& $16_{b_1} $&$ 6$ &$16_{b_1}$&$6$&$22_{b_1}$&$22_{b_1}$\\
      $24_{a_1}$& $19_{a_1}$&$5$& $9_c $&$ 5$ &$9_c$&$5$&$14_{c}$&$14_{c}$\\
      $20_{b_1}$& $16_{b_1}$&$4$& $9_{a_2} $&$ 4$ &$9_{a_2}$&$4$&$13_{a_2}$&$13_{a_2}$\\
      $12_{a_2}$& $9_{a_2}$&$3$& $3_{b_2} $&$ 3$ &$3_{b_2}$&$3$&$6_{b_2}$&$6_{b_2}$\\
      $5_{b_2}$& $3_{b_2}$&$2$& $1_c $&$ 2$ &$2_{b_4}$&$2$&$3_{c}$&$4_{b_4}$\\
      $3_{b_4}$& $2_{b_4}$&$1$& $2_{b_4} $&$ 1$ &$3_{b_3}$&$1$&$3_{b_4}$&$4_{b_3}$\\
      $3_{b_3}$& $3_{b_3}$&$0$& $3_{b_3} $&$ 0$ &$1_c$&$0$&$3_{b_3}$&$1_{c}$\\
\end{tabular}
\end{table}
\end{center}
This completes the proof.
\end{proof}

Now, a very interesting feature of the treatment in \cite{KAGEABG95} is the interplay between the bijective component, the generating function for minimal partitions, and the local generating function. Recall  that a partition $\pi$ is minimal if for any $\pi'$ with the same color sequence as $\pi$, $\sigma(\pi')\leq\sigma(\pi)\implies\pi'=\pi$. 

More precisely, in \cite{KAGEABG95} the identity
\[\sum_{i,j}\frac{a^ib^jq^{2T_i+2T_j}(-q)_{i+j}(-cq^{i+j+1})_{\infty}}{(q^2;q^2)_i(q^2;q^2)_j}
=\sum_{i,j,k}\frac{a^ib^jc^kq^{2T_i+2T_j+T_k+(i+j)k}}{(q)_{i+j+k}}{{i+j+k} \brack k}_q{{i+j} \brack i}_{q^2}\]
was understood to be the key identity for a 3-color generalization of Capparelli's partition theorem. Moreover the function 
\[H(i,j,k)=q^{2T_i+2T_j+T_k+(i+j)k}{{i+j+k} \brack k}_q{{i+j} \brack i}_{q^2},\]
appearing in the sum on the RHS was understood to be the generating function for minimal partitions with the appropriate difference conditions.

As such, the goal of the next section is to construct a framework where the function $H(i,j,k)$ can be viewed as an instance of a claim valid for a general number of colors. For simplicity, in the sequel, we will use the notational shorthand interpreting $r$-ary formulas on numbers as a $1$-ary formula on $ r$-dimensional vectors. Specifically, let\footnote{The reader wondering about the indexing notation will benefit from the reminder that in \cite{KAGEABG95} Theorem C-R was given a generalization interpreted as counting partitions in three colors $a=a_1$, $b=b_1$ and $c$, where, most relevantly, the number of occurrences of the color $a=a_1$ is $\nu_a=\nu_{a_1}=i$ and the number of occurrences of the color $b=b_2$ is $\nu_b=\nu_{b_1}=j$. Thus, this notation already foreshadows the generalization to take place by the end of the paper.} 
\[\vec{\nu}=(\nu_{a_1},\dots,\nu_{a_{r'}},\nu_{b_1},\dots,\nu_{b_{r''}}),\] then, for example, the statement of Theorem C-R now becomes \[A_3(n;\vec{v},k)=D_3(n;\vec{v},k),\]
with $\vec{v}=(\nu_{a_1},\nu_{b_1})=(i,j)$.
In addition, let $(\vec{\nu})_{z}$ denote $\vec{\nu}$ with one subtracted from the $z$ indexed coordinate. For example,
\begin{align*}
    (\vec{\nu})_{a_1}&=(\nu_{a_1}-1,\dots,\nu_{a_{r'}},\nu_{b_1},\dots,\nu_{b_{r''}}),\\
    (\vec{\nu})_{b_{r''}}&=(\nu_{a_1},\dots,\nu_{a_{r'}},\nu_{b_1},\dots,\nu_{b_{r''}}-1),
\end{align*}
and for $1<k$ let
\begin{align*}
    (\vec{\nu})_{a_k}&=(\nu_{a_1},\dots\nu_{a_{k}}-1,\dots\nu_{b_1},\dots,\nu_{b_{r''}}),\\
    \intertext{and}
    (\vec{\nu})_{b_k}&=(\nu_{a_1},\dots,\nu_{a_{r'}},\nu_{b_{1}},\dots\nu_{b_k}-1\dots).
\end{align*}

With that in mind we denote \[\sigma(\vec{\nu})=\nu_{a_1}+\dots+\nu_{a_{r'}}+\nu_{b_1}+\dots+\nu_{b_{r''}}\] and 
$2T_{\vec{\nu}}=\vec{\nu}\cdot (\vec{\nu}+\vec{1})=\sum_{i=1}^{r'}(2T_{\nu_{a_i}}+2T_{\nu_{b_i}})$ with $\vec{1}=(1,\dots,1)$. This way we can conveniently write the $q$-multinomial coefficient as  \({\sigma(\vec{\nu})\brack\vec{\nu}}={{\nu_1+\cdots+\nu_r}\brack {\nu_1, \cdots, \nu_r}}.\)

Thus, in the above notation, we can write an equivalent form to (1.8) as
\begin{align*}
 \sum_{\nu_{a_1},\dots,\nu_{b_1},\dots\nu_{b_{r''}},k\geq0}\frac{c^k\prod_{i,j} A_i^{\nu_{a_i}}B_j^{\nu_{b_i}}}{(q)_{\sigma(\vec{\nu})+k}} &q^{2T_{\vec{\nu}}+T_k+k\sigma(\vec{\nu})}{\sigma(\vec{\nu})\brack \vec{\nu}}_{q^2}\times {\sigma(\vec{\nu})+k\brack \sigma(\vec{\nu}),k}_{q} \\
 &=\sum_{\nu_{a_1},\dots,\nu_{b_1},\dots,\nu_{b_{r''}},k\geq0}\frac{q^{2T_{\vec{v}}}c^k\prod_{i,j} A_i^{\nu_{a_i}}B_j^{\nu_{b_i}} (-q)_{\sigma(\vec{v})}(-cq^{\sigma(\vec{v})+1})}{\prod_{i,j} (q^2;q^2)_{\nu_{a_i}}(q^2;q^2)_{\nu_{b_i}}}.
\end{align*}
We now partake in the analysis of \[H(\vec{\nu},k)= q^{2T_{\vec{\nu}}+T_k+k\sigma(\vec{\nu})}{\sigma(\vec{\nu})\brack \vec{\nu}}_{q^2}\times {\sigma(\vec{\nu})+k\brack \sigma(\vec{\nu}),k}_{q},\]
as a generating function for certain partitions. Unfortunately, as noted earlier, the above function does not count minimal partitions. A counterexample can already be found in the case corresponding to \cite[Thm.~C5]{YAKA25}, as will be seen shortly. However, we can produce a somewhat close interpretation.  
The proof builds on an old idea due to Y. Alamoudi (originally developed in the context of 3-color Schur), communicated to K.  Alladi by email in March 2023, that bypasses the method of weighted words and proceeds directly to the generating function for partitions with a prescribed number of parts. Observe that the above bijection takes bipartitions $(\pi_{ab},\pi_c)$ with $\pi_{ab}$ composed of distinct even non-$c$-parts counted by $\vec{v}$, and $\pi_c$ composed of distinct $c$-parts such that the number of parts of weight\footnote{Here and in the sequel, the word ``weight" refers to the underlying integer. For example, $3_c$ has weight $3$.} $>\sigma(\vec{v})$ is exactly $k$. Thus, if $G(\vec{v},k)=G(\nu_{a_1},\dots,\nu_{a_{r'}},\nu_{b_1},\dots,\nu_{b_{r''}},k)$ is the generating function for the partitions appearing in step 1 is: 
\[=\underbrace{\underbrace{\frac{q^{2T_{\nu_{a_1}}}}{ (q^2;q^2)_{\nu_{a_1}}}}_{\substack{\text{generating function for $\nu_{a_1}$}\\ \text{ distinct even $a_1$ parts} }}\cdots\frac{q^{2T_{\nu_{a_i}}}}{ (q^2;q^2)_{\nu_{a_i}}}\cdots
\frac{q^{2T_{\nu_{b_j}}}}{ (q^2;q^2)_{\nu_{b_j}}}\dots\frac{q^{2T_{\nu_{b_{r''}}}}}{ (q^2;q^2)_{\nu_{b_{r''}}}}
}_\text{Generating function for partitions composed of distinct even non-$c$-parts counted by $\vec{v}$}\underbrace{(-q)_{\sigma(\vec{v})}}_{\substack{{\text{generating function}}\\ \text{for distinct}\\ \text{ $c-$parts $\leq\sigma(\vec{v})$}}}\underbrace{\frac{q^{k\sigma(\vec{v})+T_k}}{(q)_k}}_{\substack{{\text{generating function}}\\ \text{for $k$ distinct}\\ \text{ $c-$parts $>\sigma(\vec{v})$}}}\]
canceling gives 
$$G(\vec{v},k)= \frac{1}{(q)_{\sigma(\vec{v})+k}}q^{2T_{\vec{\nu}}+T_k+k\sigma(\vec{\nu})}{\sigma(\vec{\nu})\brack \vec{\nu}}_{q^2}\times {\sigma(\vec{\nu})+k\brack \sigma(\vec{\nu}),k}_{q}.$$
Now, why can't we remove the $\frac{1}{(q)_{\sigma(\vec{v})+k}}$ factor and get the generating function for the minimal partitions? The issue is that \textit{the supplementary condition} makes it so that not every partition can be generated by adding (embedding) an arbitrary partition to a minimal partition. For example $\pi=5_c+2_{a_2}+2_{a_1}+2_{b_2}+2_{b_1}$ is a minimal partition. However, whatever partition you add to $\pi$, you will never be able to get the gap between the largest parts to shrink below $3$. Thus, there does not exist a $\pi'$ such that $\pi+\pi'=\pi''=5_c+4_{a_2}+2_{a_1}+2_{b_2}+2_{b_1}$. However, $\pi ''$ certainly satisfies the difference conditions. In fact, for $\vec{v}=(1,1,1,1)$ and $k=1$ we have
\[q^{13}{4\brack 1,1,1,1}_{q^2}\times {5\brack 4,1}_{q}=q^{13}+q^{14}+\dots\]
but it can be shown that there are no minimal partitions with $\sigma(\pi)=14$.

In the next section, we will relate this to what we call \textit{mock-minimal} partitions (defined in the next section). In particular, we will first interpret \[H(\vec{\nu},k)= q^{2T_{\vec{\nu}}+T_k+k\sigma(\vec{\nu})}{\sigma(\vec{\nu})\brack \vec{\nu}}_{q^2}\times {\sigma(\vec{\nu})+k\brack \sigma(\vec{\nu}),k}_{q}\]
as counting bipartitions $(\pi_{ab},\pi_c)$ with $\pi_{ab}$ a minimal partition composed of distinct even non-$c$-parts counted by $\vec{v}$, and $\pi_c$ composed of $k$ distinct $c$-parts each greater than $\sigma(\vec{v})$ and at most equal to $2\sigma(\vec{v})+k$. Then, we will further refine the result to count \textit{mock-minimal} partitions.  However, we will see that, in fact, much more is true. Namely, we will give formulas $H_x(\vec{v},k)$ for mock-minimal partitions whose smallest part has color $x$. In fact, just like the classical case, we will see that $H_c(\vec{v},k)$ is expressed as a sum of $r=r'+r''$ terms. Furthermore, we will give an elegant interpretation of these terms and shed new light, revealing the combinatorial meaning of the two terms appearing in $H_c(i,j,k)$ from \cite{KAGEABG95}, which seems to have escaped attention.

\section{Mock-minimal partitions and mock local generating functions}
Since we are concerned with mock-minimal partitions as opposed to minimal partitions, we will employ novel techniques that significantly diverge from the classical treatment in \cite{KAGEABG95}. As a consequence, the treatment here is a bit technical, and the reader may benefit from first reading the statement of Theorem 4, the ``Illustrations of Theorem 4" paragraph, and the closing paragraph to motivate the objectives of this section.\footnote{Comparisons with \cite[Sec.~3]{KAGEABG95} and \cite[Sec.~5]{KAGEABG95G} will also add context.}

The following lemma will be useful in facilitating the definition of a mock-minimal partition.

\begin{lemma}
Let $\pi_m$ be a minimal partition (possibly having $c$-parts). Let $\pi_c$ be composed of $k$ distinct $c$-parts. If $\max{\pi_c}-k\leq \max{\pi_m}$, the four-step process (steps 4 to 7) sends $\pi_c/\pi_m$ (obtained from bipartitions $(\pi_m,\pi_c)$) into a minimal partition.
\end{lemma}
\begin{remark}
In the proof below, the term ``progression" is used to refer to a sequence of three consecutive parts in the resulting partition. \textbf{These parts do not have to form an arithmetic progression}.
\end{remark}

\begin{proof}
We prove this by induction on $k$. The base case is obvious, since if $k=0$ nothing needs to be done. Suppose this is true up to $K$. We claim it is true for $K+1$. Now, perform the four-step\footnote{One may wish to distinguish between $\phi^u$ and $\phi^d$ where we use either upwards or downwards insertions. However, as we have seen in the last section, these are equinumerous and so no distinction is necessary.} process with the largest part of $\pi_c$ removed (i.e., on\footnote{As before, $\pi'_c/\pi_m$ is the vertical column with the parts of $\pi'_c$ in decreasing order appended underneath by the parts of $\pi_m$ in decreasing order. Refer to the last sentence in Step 3 of the proof in Section 2; see also Table 2.} $\pi'_c/\pi_m$ with\footnote{Here and throughout, as the partitions in question have distinct parts, we identify partitions with the set of their parts. In doing so we freely predicate on partitions in the language of set theory.} $\pi'_c=\pi_c\setminus\max{\pi_c}$ ) to obtain $\phi(\pi'_c/\pi_m)$.\footnote{This is to say, we start by initially placing every member of $\pi'_c$ vertically in decreasing order atop the largest member of $\pi_m$, followed by the other members of $\pi_m$ underneath in decreasing order and denote this initial configuration by $\pi'_c/\pi_m$, then $\phi(\pi'_c/\pi_m)$ is the resulting partition after the four-step process terminates.}
Observe that, by induction, $\phi(\pi'_c/\pi_m)$ is minimal. Moreover, notice that $\phi(\pi_c/\pi_m)$ agrees with $\phi(\pi'_c/\pi_m)$ on every part smaller than $\phi(\max{\pi_c}):=\max\{n_c:\ n_c\in\phi(\pi_c/\pi_m)\}$, so these parts must form a minimal partition by induction. Furthermore, every part larger than $\phi(\max{\pi_c})$ only increases by 1; therefore, if a violation of minimality occurs, it must do so exactly where $\phi(\max{\pi_c})$ was inserted. If $\phi(\max{\pi_c})=1_c$ then clearly $K+1=1$ and $\max{\pi_c}=\sigma(\vec{v})+1$ and $\min{\pi_m}$ is an $a$-type part of weight $=2$ and so the image is minimal. If $\phi(\max{\pi_c})=\max{\pi_c}$, then clearly $\max{\pi_c}$ has weight the unique integer $W$ such that $(W-k-1)_c\leq \max{\pi_m}\leq(W-k)_c$ and hence the gap is minimal (because $(W-k-1)_c\leq \max{\pi_m}$).\\

Otherwise, let $X\in\phi(\pi_c/\pi_m)$ be the largest part less than $\phi(\max{\pi_c})$ and $Y\in\phi(\pi_c/\pi_m)$ be the smallest part greater than $\phi(\max{\pi_c})$. If $\phi(\pi_c/\pi_m)$ was not minimal, then it must be the case that $X<\phi(\max{\pi_c})<Y$ violates the gap conditions. Notice that by induction, $X<Y-1$ is minimal.\\

We claim that $X$ and $Y-1$ must have different weights. Otherwise, $X=M_\beta$ and  $Y-1=M_\alpha$ but since we have the insertion $X<\phi(\max{\pi_c})<Y$ we must also have $X<\phi(\max{\pi_c})-1<Y-2 \implies X<Y-2\implies M_\beta<(M-1)_\alpha$. This can only happen if $\alpha=a_i$ and $\beta=b_j$, but since $\phi(\max{\pi_c})$ is a $c$-part $M_{b_j}<\phi(\max{\pi_c})-1 \leftrightarrow (M-1)_{a_i}<\phi(\max{\pi_c})-1$ as $c$-parts can't distinguish different parts of the same weight and color class. Thus, $\max{\pi_c}$ was inserted between consecutive parts of different weights.\\

There are two cases.
\begin{case}[$\max{\pi_c}$ was inserted between $M_x<(M+2)_y$ for some with $x\not=c\not=y$ with $x$ and $y$ of the same color type]
In this case, $\phi(\max{\pi_c})=M_c$ if $x$ is a $b$-part, with the resulting progression being $M_x<M_c<(M+3)_y$ which is minimal; $\phi(\max{\pi_c})=(M+1)_c$ if $x$ is an $a$-part with the resulting progression being $M_x<(M+1)_c<(M+3)_y$ which is again minimal.
\end{case}

\begin{case}[$\max{\pi_c}$ was inserted between $M_c<(M+1)_y$ for some with $x\not=c$]

Say that $\phi(\max{\pi_c})=M'_c$ with the resulting progression being $M_c<M'_c<(M+2)_y$. Then, since the $c$-parts were originally distinct, $M'\geq M+1$. Suppose that $M'= M+d$ with $d>1$. Then, the resulting progression is $M_c<(M+d)_c<(M+2)_y$ which implies $M_c<(M+d-1)_c<M_y$ with $d\not=1$ contradicting $M_c<(M+1)_y$ being a minimal gap from the induction hypothesis. Hence, $M'=M+1$ and the resulting progression is $M_c<(M+1)_c<(M+2)_y$, which is clearly minimal as the last gap condition is unchanged.
\end{case}
All three cases above are minimal, and hence the resulting partition is minimal, as was to be proved.
\end{proof}

\begin{corollary}
Let $(\pi_{m},\pi_c)$ be a bipartition with $\pi_m$ a minimal partition, and $\pi_c$ composed of $k$ distinct $c$-parts each greater than $\sigma(\vec{v})$ and at most equal to $2\sigma(\vec{v})+k$. Then, either $\pi_c$ is empty or there exists a $c$-part $n_c\in\phi(\pi_c/\pi_m)$ such that $\{\lambda\in\phi(\pi_c/\pi_m):\lambda<n_c\}$ is a minimal partition and $\{\lambda\in\phi(\pi_c/\pi_m):\lambda\geq n_c\}$ only contains $c$-parts.
\end{corollary}
\begin{proof}
Observe that the $c$-parts $\lambda>|\{\lambda'\in\pi_c:\lambda'\leq \lambda\}|+\max{\pi_{m}}$ are fixed by $\phi$, and the rest are sent to a minimal partition by the lemma.
\end{proof}

\textbf{Definition of mock-minimal partitions:} We can now define what a \textit{mock-minimal} partition is.
A partition $\pi$ into colored integers (with $k$ $c$-parts and non-$c$-parts counted by $\vec{v}$) is mock-minimal if there exists a bipartition $(\pi_{ab},\pi_c)$ (called an associate) with $\pi_{ab}$ a minimal partition composed of distinct even non-$c$-parts counted by $\vec{v}$, and $\pi_c$ composed of $k$ distinct $c$-parts each greater than $\sigma(\vec{v})$ and at most equal to $2\sigma(\vec{v})+k$ such that \textbf{one} of the following three statements holds. \begin{enumerate}[label=(\roman*)]
    \item If the associate $(\pi_{ab},\pi_c)$ has $\min\pi_{ab}$ an $a$-part, then $\pi=:\phi(\pi_{ab}\cup\pi_c)$. 
    \item $\pi=\phi(\pi_{ab}\cup\pi_c)$ \textbf{with $\min\pi_{ab}$ a $b$-part and $\max\pi_c<2\sigma(\vec{v})+k$}   
    \item If $\min\pi_{ab}$ is a $b$-part and $2\sigma(\nu)+k\in\pi_c$ then the associated mock-minimal partition is $\pi=\phi(\pi_c'/\hat{\pi})$ where $\pi'_c=\{(n+1)_c:n_c\in\pi_c \text{      and   } n_c<(\sigma(\vec{v})+k)_c\}$  and\\ $\hat{\pi}=\{(n+2)_x:n_x\in\pi_{ab} \}\cup\{1_c\}$.
 \end{enumerate}

Observe that mock-minimal partitions can contain no progression $m'_x<n_c<m_y$ with $y$ a color of higher order than $x$; furthermore, any minimal partition not containing the kind of progression mentioned above is mock-minimal. Also note that, although for any mock-minimal partition $\pi$ there is a part $\lambda$ such that the members of $\pi$ less than or equal to $\lambda$ form a minimal partition, $\pi$ itself may initially have large $c$-parts. However, every non-$c$-part is part of a minimal partition and any $c$-part that is not larger than every non-$c$-part must also be part of a minimal partition. Lastly, note that it follows straight from the definition that mock-minimal partitions counted by $(\vec{\nu},k)$ are in bijection with bipartitions $(\pi_{ab},\pi_c)$ with $\pi_{ab}$ a minimal partition composed of distinct even non-$c$-parts counted by $\vec{v}$, and $\pi_c$ composed of $k$ distinct $c$-parts each greater than $\sigma(\vec{v})$ and at most equal to $2\sigma(\vec{v})+k$. \\

\underline{Examples and non-examples of mock-minimal partitions:}\\
The previously mentioned example $5_c+2_{a_2}+2_{a_1}+2_{b_2}+2_{b_1}$ is both a mock-minimal and a minimal partition, whereas  $5_c+4_{a_2}+2_{a_1}+2_{b_2}+2_{b_1}$ is neither. The key thing to not here is that $2_{a_2}+2_{a_1}+2_{b_2}+2_{b_1}$ is a minimal partition whereas $4_{a_2}+2_{a_1}+2_{b_2}+2_{b_1}$ is not. On the other hand, $6_c+2_{a_2}+2_{a_1}+2_{b_2}+2_{b_1}$ is a mock-minimal partition but not a minimal partition. Note, that $6_c+4_{a_2}+2_{a_1}+2_{b_2}+2_{b_1}$ is again neither because it still contains $4_{a_2}+2_{a_1}+2_{b_2}+2_{b_1}$. Lastly, note that $5_{a_2}+5_{a_1}+3_c+2_{b_2}+2_{b_1}$ is a minimal partition but not a mock-minimal partition. Note that in the last example, the partition really satisfies all the conditions of both $G$ and $G'$ and it is genuinely minimal. The issue here that makes it not mock minimal is the progression $5_{a_1}+3_c+2_{b_2}$, in which a $c$-part is sandwiched between two parts, with the smaller having a lower color order than the larger part. Why is this an issue? This is really the content of Lemma 1, but let us illustrate the issue with this explicit example. Suppose that you get $5_{a_2}+5_{a_1}+3_c+2_{b_2}+2_{b_1}$ from $(\pi_{ab},\pi_c)$ with $\pi_{ab}$ a minimal partition composed of distinct even non-$c$-parts via the four-step process. If we now reverse the four-step process, we will see that $5_{a_2}+5_{a_1}+3_c+2_{b_2}+2_{b_1}$ comes from $(4_{a_2}+4_{a_1}+2_{b_2}+2_{b_1},5_c)$. However, $4_{a_2}+4_{a_1}+2_{b_2}+2_{b_1}$ is not minimal. This is a contradiction, and we see that $5_{a_2}+5_{a_1}+3_c+2_{b_2}+2_{b_1}$ cannot be obtained by the process describing mock-minimal partitions.

\begin{corollary}
     \[H(\vec{\nu},k)= q^{2T_{\vec{\nu}}+T_k+k\sigma(\vec{\nu})}{\sigma(\vec{\nu})\brack \vec{\nu}}_{q^2}\times {\sigma(\vec{\nu})+k\brack \sigma(\vec{\nu}),k}_{q}\tag{4.1}\]
is the generating function for mock-minimal partitions where there are $k$ $c$-parts and $\nu_x$ parts of color $x\in\{a_1,\dots,a_{r'},b_1,\dots,b_{r''}\}$. 
\end{corollary}

Now, let $H_{x}(\vec{\nu},k)$ be the generating function for mock-minimal partitions with the prescribed color parameters whose smallest part has the color $x$. In addition, for $t\not=c$, let $H_c(\vec{\nu},k)_{t}$ be the generating function for mock-minimal partitions with the prescribed color parameters, minimum part having the color $c$, and smallest non-$c$-part having color $t$.
\begin{theorem}
\begin{align*}
\intertext{For each $1\leq j\leq r'$ we have}
    H_{a_j}(\vec{\nu},k)&=q^{2T_{\vec{\nu}}+T_k+k\sigma(\vec{\nu})+k+2\sigma_b(\vec{v})+2\sum_{i<j}\nu_{a_i}}{\sigma(\vec{\nu})-1\brack (\vec{\nu})_{a_j}}_{q^2}\times {\sigma(\vec{\nu})+k-1\brack \sigma(\vec{\nu})-1,k}_{q}\tag{4.2a}\\
\intertext{and}
H_c(\vec{\nu},k)_{a_j}&=\:\:q^{2T_{\vec{\nu}}+T_k+k\sigma(\vec{\nu})+2\sigma_b(\vec{v})+2\sum_{i<j}\nu_{a_i}}{\sigma(\vec{\nu})-1\brack (\vec{\nu})_{a_j}}_{q^2}\times {\sigma(\vec{\nu})+k-1\brack \sigma(\vec{\nu}),k-1}_{q}.\tag{4.2ca}
\intertext{In a similar way, for each $1\leq j\leq r''$ we have}    
    H_{b_j}(\vec{\nu},k)&=\:\:\:q^{2T_{\vec{\nu}}+T_k+k\sigma(\vec{\nu})+2\sum_{i<j}\nu_{b_i}}\:\:\:\:\:{\sigma(\vec{\nu})-1\brack (\vec{\nu})_{b_j}}_{q^2}\times {\sigma(\vec{\nu})+k-1\brack \sigma(\vec{\nu})-1,k}_{q}\tag{4.2b}\\
\intertext{and}
    H_c(\vec{\nu},k)_{b_j}&=\:\:q^{2T_{\vec{\nu}}+T_k+(k+1)\sigma(\vec{\nu})+2\sum_{i<j}\nu_{b_i}}{\sigma(\vec{\nu})-1\brack (\vec{\nu})_{b_j}}_{q^2}\times {\sigma(\vec{\nu})+k-1\brack \sigma(\vec{\nu}),k-1}_{q}.\tag{4.2cb}
\intertext{In particular, it follows that the local generating function for minimal partitions having a $c$-part as the smallest part is}     
    H_c(\vec{\nu},k)\:&=\:\:\:\:\bigg\{\sum_{i=1}^{r'} H_c(\vec{\nu},k)_{a_i}\bigg\}\:\:\:\:+\:\:\:\:\sum_{i=1}^{r''} H_c(\vec{\nu},k)_{b_i}.\tag{4.2c}\\
\end{align*}
\end{theorem}

To establish Theorem 3, we begin by recalling the well-known, see \cite{qm} and \cite{KAGEABG95G}, multinomial coefficient recurrence.

\begin{align*}
  \tag{4.MR}  {\sigma(\vec{\nu})\brack \vec{\nu}}_{q^2}&={\sigma(\vec{\nu})-1\brack (\vec{\nu})_{b_1}}_{q^2} +&\dots q^{2(\nu_{b_1}+\dots\nu_{b_{j-1}})}{\sigma(\vec{\nu})-1\brack (\vec{\nu})_{b_j}}_{q^2}&\\
    &&\dots+q^{2\sigma_b(\vec{\nu})}{\sigma(\vec{\nu})-1\brack (\vec{\nu})_{a_1}}_{q^2}+&\dots& q^{2\sigma_b(\vec{\nu})+2\sum_{i<r'}\nu_{a_i}}{\sigma(\vec{\nu})-1\brack (\vec{\nu})_{a_{r'}}}_{q^2}\\
\end{align*}
\[\hspace*{-1cm}=\bigg\{\sum_{j=1}^{r'}q^{2\sum_{i<j}\nu_{b_i}}{\sigma(\vec{\nu})-1\brack (\vec{\nu})_{b_j}}_{q^2}\bigg\}+\sum_{j=1}^{r''}q^{2\sigma_b(\vec{v})+2\sum_{i<j}\nu_{a_i}}{\sigma(\vec{\nu})-1\brack (\vec{\nu})_{a_j}}_{q^2}
\]

\begin{lemma}
\begin{align*}
\intertext{For each $1\leq j\leq r'$ we have}
    H_{a_j}(\vec{\nu},0)&=\,q^{2T_{\vec{\nu}}+2\sigma_b(\vec{v})+2\sum_{i<j}\nu_{a_i}}{\sigma(\vec{\nu})-1\brack (\vec{\nu})_{a_j}}_{q^2}.\tag{4.2a-Int}\\
\intertext{In a similar way, for each $1\leq j\leq r''$ we have}    
    H_{b_j}(\vec{\nu},0)&=\:q^{2T_{\vec{\nu}}+2\sum_{i<j}\nu_{b_i}}\:\:\:\:\:\:\:\:{\sigma(\vec{\nu})-1\brack (\vec{\nu})_{b_j}}_{q^2}.\tag{4.2b-Int}\\
\end{align*}
    
\end{lemma}
\begin{proof}
 We prove this by induction on $\sigma(\vec{v})$. For, $H_{x}(\vec{\nu},0)$, with $x\not=c$, notice that the smallest pair of parts is $4_y2_x$ whenever $x$ is an $a$-part and $y$ is a $b$-part or $x$ and $y$ have the same color class and $y$ has lower order. On the other hand, the smallest pair of parts is $2_y2_x$ whenever $x$ is a $b$-part and $y$ is an $a$-part or $x$ and $y$ have the same color class and $y$ has lower order. The first case corresponds to adding $2$ to every part, and the second corresponds to not doing anything at all and simply appending the new part. Therefore, for each $1\leq j\leq r'$ we have
    \[H_{a_j}(\vec{\nu},0)=q^{\sigma(\vec{v})}\bigg\{(\sum_{ i=1}^{j}H_{a_i}((\vec{\nu})_{a_j},0))+\sum_{i=1}^{r''}H_{b_i}((\vec{\nu})_{a_j},0)\bigg\}+q^2\sum_{ i>j}H_{a_i}((\vec{\nu})_{a_j},0)\tag{4.3a-Int}.\]
Likewise, for each $1\leq j\leq r''$ we have
    \[H_{b_j}(\vec{\nu},0)=q^2\bigg\{(\sum_{ i=1}^{r'}H_{a_i}((\vec{\nu})_{a_j},0))+\sum_{ i>j}H_{b_i}((\vec{\nu})_{a_j},0)\bigg\}+q^{\sigma(\vec{v})}\sum_{ i=1}^{j}H_{b_i}((\vec{\nu})_{b_j},0))\tag{4.3b-Int}.\]
By induction, the terms inside the sum have the desired form. Now notice that all the terms on the RHS of (4.3a-Int) have a factor of $q^{2T_{\vec{\nu}}+2\sigma_b(\vec{v})+2\sum_{i<j}\nu_{a_i}}$. Factor this out, then the conclusion for (4.2a-Int) follows directly from the recurrence (4.MR). Likewise, all the terms on the RHS of (4.3b-Int) have a factor of $q^{2T_{\vec{\nu}}+2\sum_{i<j}\nu_{b_i}}$, which, after factoring out and appealing to the recurrence (4.MR), proves (4.2b-Int).
\end{proof}
\begin{remark}
The above lemma can also be thought of in terms of inversions as in \cite[Prop.2.3]{qm}. This is because whenever a lower-order part is greater than a higher-order part, the gap must be at least $2$. However, in view of the ordering scheme, a lower-order part being bigger than a higher-order part can be viewed as an inversion. In other words, you get an extra factor of $q^2$ for each inversion.
\end{remark}

We are almost there. This intermediate corollary is a direct but crucial consequence of the definition of mock-minimal partitions.

\begin{corollary}
For $x$ a non-$c$-part, let $H_{x\lor c}(\vec{\nu},k)$ be the generating function for mock-minimal partitions whose smallest non-$c$-part has color $x$. We have the following.  
\begin{align*}
\intertext{For each $1\leq j\leq r'$ we have}
    H_{a_j\lor c}(\vec{\nu},k)=H_{a_j}(\vec{\nu},0)& \times q^{T_{k}+k(\vec{\nu})}{\sigma(\vec{\nu})+k\brack \sigma(\vec{\nu}),k}_{q}.\tag{4.4a}\\
\intertext{In a similar way, for each $1\leq j\leq r''$ we have}    
    H_{b_j\lor c}(\vec{\nu},k)=H_{b_j}(\vec{\nu},0)&\times q^{T_{k}+k(\vec{\nu})}{\sigma(\vec{\nu})+k\brack \sigma(\vec{\nu}),k}_{q}.\tag{4.4b}\\
\end{align*}
\end{corollary}
\begin{proof}
First observe that $q^{T_{k}+k(\vec{\nu})}{\sigma(\vec{\nu})+k\brack \sigma(\vec{\nu}),k}_{q}$ counts partitions $\pi_c$ composed of $k$ distinct $c$-parts each greater than $\sigma(\vec{v})$ and at most equal to $2\sigma(\vec{v})+k$. The mock-minimal partitions that begin with an $a_j$ \textbf{or} $c$ are accounted for by (i) in the definition of mock-minimal partitions and are therefore in bijection with $(\pi_{ab},\pi_c)$ where $\pi_{ab}$ is a minimal partition, whose smallest part is $a_j$, having parts counted by $\vec{\nu}$ and $\pi_c$ composed of $k$ distinct $c$-parts each greater than $\sigma(\vec{v})$ and at most equal to $2\sigma(\vec{v})+k$. From Lemma 2 and the observation at the beginning of this proof, we conclude that (4.4a) is true. 

In a similar way, the mock-minimal partitions whose smallest part is $b_j$ are accounted for by (ii) in the definition of mock-minimal partitions. Additionally, the mock-minimal partitions whose smallest is a $c$-part and whose smallest non-$c$-part is $b_j$ are accounted for by (iii) in the definition of mock-minimal partitions. Thus, they are in bijection with $(\pi_{ab},\pi_c)$ where $\pi_{ab}$ is a minimal partition, whose smallest part is $b_j$, having parts counted by $\vec{\nu}$ and $\pi_c$ composed of $k$ distinct $c$-parts each greater than $\sigma(\vec{v})$ and at most equal to $2\sigma(\vec{v})+k$. From Lemma 2 and the observation at the beginning of this proof, we conclude that (4.4b) is true. 
\end{proof}

We can now prove Theorem $3$. This proof is versatile enough to apply to other problems. In particular, the argument will be used to shed light on the classical $H_c(i,j,k)$ from \cite{KAGEABG95}.

\begin{proof}[Proof of Theorem 3]
We proceed by induction. The base case is trivial. Suppose this is true up to $\sigma(\vec{\nu})+k\leq N-1$. We begin by $H_c(\vec{\nu},k)_x$ with $x$ a non-$c$-part.\medskip

If $x=a_j$ with $1\leq j\leq r'$, then we can obtain a mock-minimal partition having the smallest part a $c$-part and the smallest non-$c$-part an $a_j$-part by appending a string of small $c$-parts $k'_c+\dots+2_c+1_c$, for each $1\leq k'\leq k$, and increasing every part of the remaining partition by an appropriate amount. This gives   
\begin{align*}
    H_c(\vec{\nu},k)_{a_j}=&\sum_{i=1}^k q^{i(\sigma(\vec{\nu})+k)-T_{i-1}}H_{a_j}(\vec{\nu},k-i)\\&=q^{2T_{\vec{\nu}}+T_k+k\sigma(\vec{\nu})+2\sigma_b(\vec{v})+2\sum_{i<j}\nu_{a_i}}{\sigma(\vec{\nu})-1\brack (\vec{\nu})_{a_j}}_{q^2}\times \sum_{i=1}^kq^{k-i}{\sigma(\vec{\nu})-1+k-i\brack \sigma(\vec{\nu})-1,k-i}_{q}\\
    &=\:\:q^{2T_{\vec{\nu}}+T_k+k\sigma(\vec{\nu})+2\sigma_b(\vec{v})+2\sum_{i<j}\nu_{a_i}}{\sigma(\vec{\nu})-1\brack (\vec{\nu})_{a_j}}_{q^2}\times {\sigma(\vec{\nu})+k-1\brack \sigma(\vec{\nu}),k-1}_{q}
\end{align*}
In view of the $q$-hockey stick identity (see for example \cite[Thm.~3.4]{GEA}), the follows for (4.2ca).\\  

In exactly the same way, if $x=b_j$ with $1\leq j\leq r''$, then we can obtain a mock-minimal partition having the smallest part a $c$-part and the smallest non-$c$-part a $b_j$-part by appending a string of small $c$-parts $k'_c+\dots+2_c+1_c$, for each $1\leq k'\leq k$, and increasing every part of the remaining partition by an appropriate amount as described by (ii) in the definition of mock-minimal partitions. This gives \[\] 
\begin{align*}
    H_c(\vec{\nu},k)_{b_j}=&\sum_{i=1}^k q^{\sigma(\vec{\nu})+i(\sigma(\vec{\nu})+k)-T_{i-1}}H_{b_j}(\vec{\nu},k-i)\\&=q^{2T_{\vec{\nu}}+T_k+(k+1)\sigma(\vec{\nu})+2\sum_{i<j}\nu_{b_i}}{\sigma(\vec{\nu})-1\brack (\vec{\nu})_{b_j}}_{q^2}\times \sum_{i=1}^kq^{k-i}{\sigma(\vec{\nu})-1+k-i\brack \sigma(\vec{\nu})-1,k-i}_{q}\\
    &=\:\:q^{2T_{\vec{\nu}}+T_k+(k+1)\sigma(\vec{\nu})+2\sum_{i<j}\nu_{b_i}}{\sigma(\vec{\nu})-1\brack (\vec{\nu})_{b_j}}_{q^2}\times {\sigma(\vec{\nu})+k-1\brack \sigma(\vec{\nu}),k-1}_{q}.
\end{align*}
All that remains to show the theorem is to note that for $x$ a non-$c$-part \[H_x(\vec{\nu},k)=H_{x\lor c}(\vec{\nu},k)-H_{c}(\vec{\nu},k)_x\]
and invoke the previous corollary.
\end{proof}
We now give a very elegant application to the classic case.  Recall that in \cite{KAGEABG95} it was shown that \[H_c(i,j,k)=q^{2T_i+2T_j+T_k+(i+j)k}\bigg\{q^{i+j}{i+j+k-1\brack i+j,k-1}_q{i+j-1\brack i,j-1}_{q^2}+q^{2j}{i+j+k-1\brack i+j,k-1}_q{i+j-1\brack i-1,j}_{q^2}\bigg\}\]
where $H_c(i,j,k)$ counts minimal partitions that begin with a $c$-part. We have the following corollary.

\begin{corollary}
Let $H_c(i,j,k)_x$ be the generating function for minimal partitions whose smallest part is a $c$-part and smallest non-$c$ part an $x$-part. Then, we have the following. 
\[H_c(i,j,k)_a=q^{2T_i+2T_j+T_k+(i+j)k+2j}{i+j+k-1\brack i+j,k-1}_q{i+j-1\brack i-1,j}_{q^2}\]
\[H_c(i,j,k)_b=q^{2T_i+2T_j+T_k+(i+j)(k+1)}{i+j+k-1\brack i+j,k-1}_q{i+j-1\brack i,j-1}_{q^2}\]
\end{corollary}
\begin{proof}
The proof is the same as the last proof we did for Theorem 3. The only difference is that when we are inducting, the previous cases are generating functions for minimal partitions and not just mock-minimal partitions. For the $a$-parts, we are increasing every part by one every time we add a $c$ part. Like before, this amounts to 
\begin{align*}
    H_c(i,j,k)_{a}=&\sum_{i'=1}^k q^{i'(i+j+k)-T_{i'-1}}H_{a}(i,j,k-i')\\&=q^{2T_i+2T_j+T_k+(i+j)k+2j}{i+j-1\brack i-1,j}_{q^2}\times \sum_{i'=1}^kq^{k-i'}{i+j-1+k-i'\brack i+j-1,k-i'}_{q}\\
    &=q^{2T_i+2T_j+T_k+(i+j)k+2j}{i+j+k-1\brack i+j,k-1}_q{i+j-1\brack i-1,j}_{q^2}.
\end{align*}
For the $b$-parts, it is the same except that the \textbf{first time you add a $c$-part, you increase every part by 2} (this is the combinatorial reason for the extra $q^{i+j}$ factor). After, we are increasing every part by one every time we add a $c$ part. This amounts to 
\begin{align*}
    H_c(i,j,k)_{b}=& \sum_{i'=1}^k q^{i+j+k+i'(i+j+k)-T_{i'-1}}H_{b}(i,j,k-i')\\&=q^{2T_i+2T_j+T_k+(i+j)(k+1)}{i+j-1\brack i-1,j}_{q^2}\times \sum_{i'=1}^kq^{k-i'}{i+j-1+k-i'\brack i+j-1,k-i'}_{q}\\
    &=q^{2T_i+2T_j+T_k+(i+j)(k+1)}{i+j+k-1\brack i+j,k-1}_q{i+j-1\brack i,j-1}_{q^2}
\end{align*}
\end{proof}

Now, a nice thing about the classic generating functions $H_a(i,j,k), H_b(i,j,k)$, and $H_c(i,j,k)$ is multiplying by $\frac{1}{(q)_{i+j+k}}$ gives the generating function for all partitions with the prescribed number of parts and smallest part having color $a$, $b$, and $c$ respectively. This is classically obtained by embedding a partition counted by $\frac{1}{(q)_{i+j+k}}$ into a minimal partition. However, in view of the key counterexamples discussed at the end of the last section, we already know that we can't obtain all partitions counted by the higher order $H$ by this method, nor can we even decompose an arbitrary partition counted by $D$ or $D'$ to a minimal partition satisfying the relevant different conditions and a partition into at most $\sigma(\vec{v})+k$ parts. What is remarkable is that, despite the classic justification completely failing, multiplying by $\frac{1}{(q)_{\sigma(\vec{v})+k}}$ gives a perfect analogy to the classic scenario. In particular, despite the fact that mock-minimal partitions greatly diverge from the classical minimal partitions, they end up behaving in exactly the same way when multiplied by $\frac{1}{(q)_{\sigma(\vec{v})+k}}$. This is the content of the next theorem.

\begin{theorem}
Let $G_{x}(\vec{\nu},k)$ be the generating function for partitions satisfying the conditions of $D$ or $D'$ whose smallest part is an $x$-part. In addition, for $t\not=c$, let $G_c(\vec{\nu},k)_{t}$ be the generating function for partitions satisfying the conditions of $D$ or $D'$ whose smallest part is a $c$-part and smallest non-$c$-part is a $t$-part. Then,
\[G_{x}(\vec{\nu},k)=\frac{H_{x}(\vec{\nu},k)}{(q)_{\sigma(\vec{\nu})+k}}\]
and \[G_{c}(\vec{\nu},k)_t=\frac{H_{c}(\vec{\nu},k)_t}{(q)_{\sigma(\vec{\nu})+k}}.\]
\end{theorem}
In view of the key counterexamples, we already know that the proof of the above theorem cannot be the same as the classical embedding. In fact, the proof will be analogous to what was done to the corresponding $H$ functions. 
\begin{lemma}
For $x$ a non-$c$-part, let $G_{x\lor c}(\vec{\nu},k)$ be the generating function for satisfying the conditions of $D$ or $D'$ with smallest non-$c$-part having color $x$. Then, $G_{x\lor c}(\vec{\nu},k)=\frac{H_{x\lor c}(\vec{\nu},k)}{(q)_{\sigma(\vec{\nu})+k}}$.  
\end{lemma}
\begin{proof}
In view of equation 4a and 4b we already have \[\frac{H_{x\lor c}(\vec{\nu},k)}{(q)_{\sigma(\vec{\nu})+k}}=\frac{H_{x}(\vec{\nu},0)}{(q)_{\sigma(\vec{\nu})}}\times \frac{q^{T_{k}+k(\vec{\nu})}}{(q)_k}.\]
Now, $\frac{q^{T_{k}+k(\vec{\nu})}}{(q)_k}$ can be plainly interpreted as the generation function into $k$ $c$-parts each greater than $\sigma(\vec{\nu})$. On the other hand, as mock-minimal partitions coincide with regular minimal partitions when $k=0$, we can simply embed a partition into at most $\sigma(k)$ parts into a minimal partition counted by $H_x(\vec{\nu},0)$. In other words, the classic embedding already establishes Theorem 4 in the case $k=0$. What we have so far is that $\frac{H_{x\lor c}(\vec{\nu},k)}{(q)_{\sigma(\vec{\nu})+k}}$ counts all bipartitions $(\pi_{ab},\pi_c)$ where $\pi_{ab}$ has $\sigma(\vec{\nu})$ parts non of which is a $c$-part and $\pi_c$ is composed of $k$ $c$-parts each greater than $\sigma(\vec{\nu})$. This is in bijection with the partitions counted by $G_{x\lor c}(\vec{\nu},k)$ via $(\pi_{ab},\pi_c)\to \phi(\pi_{ab}\cup\pi_c)$. This completes the proof. 
\end{proof}

We now prove Theorem 4 by induction. In view of the above lemma, we will focus on showing the second equation in Theorem 4, which will play the biggest role in the inductive step. This is because the first equation will directly follow from $G_{x}(\vec{\nu},k)=G_{x\lor c}(\vec{\nu},k)-G_{c}(\vec{\nu},k)_x$ which is obvious.

\begin{proof}[Proof of Theorem 4.]
In view of the last remark, it suffices to show that $G_{c}(\vec{\nu},k)_t=\frac{H_{c}(\vec{\nu},k)_t}{(q)_{\sigma(\vec{\nu})+k}}$. We proceed by induction on $k$. Note that we have $G_{c}(\vec{\nu},0)_t=\frac{H_{c}(\vec{\nu},0)_t}{(q)_{\sigma(\vec{\nu})}}=0$ whenever $\vec{v}\not=\vec{0}$ and $k>0$.  Suppose that the claim of the theorem is true up to $K$. We first note that if $x=a_j$, then for $i>0$ $ q^{i(\sigma(\vec{\nu})+K)-T_{i-1}}G_{a_j}(\sigma(\vec{\nu}),K-i)$ is the generating function for partitions satisfying the conditions of $D$ or $D'$ whose smallest $i$ parts are $i_c+\dots+2_c+1_c$ and $(i+1)$-th smallest part an $a_j$-part. Now note that, we can bijectively obtain any partitions satisfying the conditions of $D$ or $D'$ whose smallest $i$ parts are $c$-parts and $(i+1)$-th smallest part an $a_j$-part from a partition $\pi_0$ counted by $ q^{i(\sigma(\vec{\nu})+k)-T_{i-1}}G_{a_j}(\sigma(\vec{\nu}),K-i)$ and a partition $\lambda_1+\dots+\lambda_\ell$ into parts whose size is at least $\sigma(\vec{\nu})+K-i+1$ and at most $\sigma(\vec{\nu})+K$ as follows. Let $\pi_{i+1}$ be the partition obtained from $\pi_{i}$ by adding $1$ to the largest $\lambda_1$ parts of $\pi_i$. Then, $(\pi_0,\lambda_1+\dots+\lambda_\ell)\to\pi_\ell$ is the desired (iterative) bijection. This shows that for $i>0$ $ \frac{q^{i(\sigma(\vec{\nu})+K)-T_{i-1}}}{(q^{\sigma(\vec{\nu})+K-i+1};q)_i}G_{a_j}(\sigma(\vec{\nu}),K-i)$ is the generating function for partitions satisfying the conditions of $D$ or $D'$ whose smallest $i$ parts are $c$-parts and $(i+1)$-th smallest part is an $a_j$-part. Now by induction \(G_{c}(\vec{\nu},K-i)_t=\frac{H_{c}(\vec{\nu},K-i)_t}{(q)_{\sigma(\vec{\nu})+K-i}}\) when $i>0$ and so
\begin{align*}
    G_c(\vec{\nu},K)_{a_j}&=\sum_{i=1}^K \frac{q^{i(\sigma(\vec{\nu})+K)-T_{i-1}}}{(q^{\sigma(\vec{\nu})+K-i+1};q)_i}G_{a_j}(\sigma(\vec{\nu}),K-i)\\&=\sum_{i=1}^K \frac{q^{i(\sigma(\vec{\nu})+K)-T_{i-1}}}{(q^{\sigma(\vec{\nu})+K-i+1};q)_i}\frac{H_{c}(\vec{\nu},K-i)_{a_j}}{(q)_{\sigma(\vec{\nu})+K-i}}\\
    &=\frac{H_{c}(\vec{\nu},K)_{a_j}}{(q)_{\sigma(\vec{\nu})+K}}.
\end{align*}
As we already noted, $G_{a_j}(\vec{\nu},K)=G_{a_j\lor c}(\vec{\nu},K)-G_{c}(\vec{\nu},K)_{a_j}$. The case when $x$ and $t$ are $b$ colors, and there are $K$ $c$-parts is similar. Thus, Theorem 4 is true for all $(\vec{\nu},k)$ by induction. 
\end{proof}

\underline{Illustrations of Theorem 4:}\\
For $\vec{v}=(1,1,1,1)=\vec{1}$ and $k=1$ we have that
\[G(\vec{1},1)_{b_1}=\frac{q^{13}}{(q)_5}{3\brack 1,1,1}_{q^2}\times {4\brack 3,1}_{q}=q^{13}+2q^{14}+\dots\]
is the generating counting \textbf{all} $D$ (or $D'$) partitions (not just mock-minimal) with the prescribed color parameters whose smallest part is a $b_1$-part. Thus, $5_c+2_{a_2}+2_{a_1}+2_{b_2}+2_{b_1}$ is the only such partition of sum $13$; likewise, $6_c+2_{a_2}+2_{a_1}+2_{b_2}+2_{b_1}$ and $5_c+3_{a_2}+2_{a_1}+2_{b_2}+2_{b_1}$ are the partitions of sum $14$. Note that all the aforementioned examples satisfy both $D$ and $D'$.  Now, if we take $q\to q^5$ and invoke the standard transformations\footnote{See the proof of Corollary 2.} (ultimately multiplying by a factor of  $q^{-10}$), we obtain the following generating function.
\[\frac{q^{55}}{(q^5;q^5)_5}{3\brack 1,1,1}_{q^{10}}\times {4\brack 3,1}_{q^5}=q^{55}+2q^{60}+\dots \tag{4.5}\]
Now, (4.5) counts \textbf{all} $D$ (or $D'$) partitions\footnote{This is to say, this counts those partitions satisfying Theorem 1 in the special case of \cite[Thm.~C5]{YAKA25}.} with one part from each of the five allowed congruence classes and smallest part $\equiv2 \pmod{5}$. The corresponding smallest such partition $25+8+6+4+2$; likewise, the corresponding second smallest partitions are $30+8+6+4+2$ and $25+13+6+4+2$. Again, all these partitions satisfy both $D$ and $D'$.
\bigskip

There are two takeaways from this section to be noted. Theorem 4 fully justifies the claim that the LHS of (1.8) counts the weighted words generalizations of $D_m(n;\nu_2,\dots,k)$ and $D'_m(n;\nu_2,\dots,k)$. Consequently, it really is fitting to call (1.8) a key identity of Theorem 2. In addition, we now have a computational shortcut for obtaining the number of partitions of $n$ with a prescribed smallest (and sometimes second smallest) congruence/color class. This is to say that the analysis of this section provides both conceptual insight (regarding the key identity) and computational utility (as was just illustrated). 
\section{Concluding Remarks}
Here, we have presented a very general framework extending Capparelli's partition theorem. As such, one would expect that claims about initial members of the hierarchy will extend throughout. In fact, the most famous example occurs in the first member of the hierarchy, namely, Euler's partition. Recall that Euler proved this partition theorem by writing an infinite product and cleverly exploiting the difference of squares rule from elementary algebra. One may wonder if this works throughout. Indeed, it can be seen that there is yet a fifth function $C_m(n)$ that counts partitions into parts not divisible by $4$ where the only odd parts are multiples of $m$ and the only even parts are non-multiples of $m$. Building on this, we now note the following.\\

First, the number of partitions of $n$ into parts $ \equiv\pm 2, \pm 3\pmod{12}$ is equal to the number of partitions of $n$ into odd parts, where the number of occurrences of non-multiples of $3$ must be even. 

Analogously, in the case $m=5$, the partition function for partitions into parts $2,5,6,14,15, 18\equiv \pmod{20}$, is equal to the partition function enumerating partitions into odd parts, such that the non-multiples of $5$ must occur an even number of times.

More generally, we can see that for odd $m$ there is yet a sixth companion function $C'_m(n)$ counting the number of partitions of $n$ into odd parts where the frequencies of each non-multiple of $m$ is even.\\

Interestingly, we see that just like the function $D_m(n)$ only splits into two semantically different, but equal in value, functions $D_m(n)$ and $D'_m(n)$ when $m>3$, the function $C_m(n)$ only splits into two semantically different, but equal in value, functions $C_m(n)$ and $C'_m(n)$ when $m>1$. This illustrates an added utility of the infinite hierarchy of extending initial observations to a larger domain that would have otherwise been overlooked.\\

As a last thought, given the foundational relations of Capparelli's conjecture to affine Lie algebras, we wonder if the infinite partition hierarchy introduced here fits in the framework of Lie theory.\\

{\bf{Some related works:}} Ever since the {\it{the mentod of weighted words}} was introduced by Alladi and Gordon \cite{KABG93} for a generalization and refinement of Schur's theorem, and extended by Alladi, Andrews, and Gordon to yield \cite{KAGEABG95G} refinements and generalizations of G\"ollnitz' (Big) Theorem, and Capparelli's Theorem C, this method has been applied by several authors in subsequent years either to obtain new refined partition theorems, or to extend the Schur, G\"ollnitz, and Capparelli theorems in various ways, sometimes in the form of infinite hierarchies. As major instances of such work related to, but different from, the theme of this paper, we cite the works of Dousse \cite{JD20}, Dousse and Lovejoy \cite{JDJL19}, and Dousse and Konan \cite{JDIK22}. In Dousse and Lovejoy \cite{JDJL19}, using jagged partitions, three generalizations of the undilated version of Theorem C-R are established. Dousse \cite{JD20} in 2020 showed how \cite[Thm.~3]{KAGEABG95} can be obtained as a special case of Primc's identity. Finally, Dousse and Konan \cite{JDIK22} in 2022 provided infinite families of partition identities stemming from the Capparelli and Primc identities. But the infinite hierarchies we present here differ from those given by Dousse and Konan, as well as by other authors such as Berkovich and Uncu.

\section*{Acknowledgment}
We thank both referees for their careful reading of the manuscript and for their insightful comments and constructive criticism. Incorporating their suggestions has substantially improved the clarity and exposition of this paper. We are also grateful to Referee 2 for highlighting several relevant papers and providing a comprehensive report.

\bibliographystyle{amsplain}

\end{document}